\documentclass[11pt,reqno,twoside]{amsart}
\usepackage{graphicx}
\UseRawInputEncoding 
\usepackage{subfigure}
\usepackage{amsmath,mathdots,amssymb,latexsym,amsthm,mathrsfs,epsfig}
\usepackage{multirow}
\usepackage{breqn}
\usepackage{float}
\usepackage{cite}
\usepackage{enumitem}
\usepackage{environ}
\usepackage{hyperref}
\usepackage{multicol}
\usepackage{mathtools}
\NewEnviron{myequation}{%
	\begin{equation}
		\scalebox{1.1}{$\BODY$}
	\end{equation}
}
\usepackage{color}
\usepackage{soul}
\renewcommand{\baselinestretch}{1.2}
\makeatletter
\newcommand{\single}{\let\CS=\@currsize\renewcommand{\baselinestretch}{1.1}\tiny\CS}
\newcommand{\singb}{\let\CS=\@currsize\renewcommand{\baselinestretch}{1}\tiny\CS}
\newcommand{\singa}{\let\CS=\@currsize\renewcommand{\baselinestretch}{1.2}\tiny\CS}
\newcommand{\oneandahalfspacing}{\let\CS=\@currsize\renewcommand{\baselinestretch}{1.5}\tiny\CS}
\newcommand{\singlespacing}{\let\CS=\@currsize\renewcommand{\baselinestretch}{1.6}\large\CS}
\newcommand{\bc}{\begin{center}}
	\newcommand{\ec}{\end{center}}
\newcommand{\be}{\begin{eqnarray}}
	\newcommand{\ee}{\end{eqnarray}}

\newcommand{\Hom}{\operatorname{Hom}}
\newcommand{\Ext}{\operatorname{Ext}}
\newcommand{\Dim}{\operatorname{dim}}
\newcommand{\Irr}{\operatorname{Irr}}
\newcommand{\Ind}{\operatorname{Ind}}
\newcommand{\ind}{\operatorname{ind}}
\newcommand{\diag}{\operatorname{diag}}
\newcommand{\Alg}{\operatorname{Alg}}

\newcommand{\N}{\operatorname{N}}
\newcommand{\Te}{\operatorname{T}}
\makeatletter
\newcommand{\beano}{\begin{eqnarray*}}
	\newcommand{\eeano}{\end{eqnarray*}}

\newcommand{\ba}{\begin{array}}
	\newcommand{\ea}{\end{array}}

\makeatletter

\newtheorem{thm}{Theorem}[section]

\newtheorem{lem}[thm]{Lemma}
\newtheorem{prop}[thm]{Proposition}

\theoremstyle{definition}
\newtheorem{defn}{Definition}[section]
\newtheorem{rem}{Remark}[section]
\newtheorem{exa}{Example}[section]

\numberwithin{equation}{section}
\DeclareMathOperator{\D}{D}
\DeclareMathOperator{\K}{k}
\DeclareMathOperator{\GL}{GL}
\DeclareMathOperator{\Sp}{Sp}
\DeclareMathOperator{\Rad}{Rad}
\allowdisplaybreaks
\begin{document}
	
	\title[SYMPLECTIC PERIOD FOR A REPRESENTATION OF $\GL_n(\D)$]{SYMPLECTIC PERIOD FOR A REPRESENTATION OF $\GL_n(\D)$}
	
	
	\author[HARIOM SHARMA]{HARIOM SHARMA}
	\address{Department of Mathematics, Indian Institute of Technology Roorkee, Uttarakhand, 247667, Bharat (India)}
	\email{hariom\textunderscore s@ma.iitr.ac.in}

	
	\author[MAHENDRA KUMAR VERMA]{MAHENDRA KUMAR VERMA}
	\address{Department of Mathematics, Indian Institute of Technology Roorkee, Uttarakhand, 247667,  Bharat (India)}
	\email{mahendraverma@ma.iitr.ac.in}

	\subjclass[2020]{primary 22E50; secondary 11F70.}
	\keywords{distinguished representations, division algebra,  symplectic period.}

	\begin{abstract}
		Let $\D$ be a quaternion division algebra over a non-archimedean local field $\K$ of characteristic zero, and let  $\Sp_n(\D)$ be the unique non-split inner form of the symplectic group $\Sp_{2n}(\K)$.
		This paper classifies the irreducible admissible representations of $\GL_{n}(\D)$ with a symplectic period for $n = 3$ and $4$, i.e., those irreducible admissible representations $(\pi,V)$ of $\GL_{n}(\D)$ which have a linear functional $l$ on $V$ such that $l(\pi(h)v) = l(v)$ for all $v \in V$ and $h \in \Sp_n(\D)$.
		Our results also contain all unitary representations having a symplectic period, as stated in Prasad's conjecture.
	\end{abstract}
	
	\maketitle

	\section{Introduction}
	Let $\D$ be a quaternion division algebra over a non-archimedean local field $\K$ of characteristic zero.
	Let $G = \GL_n(\D)$ and $H = \Sp_n(\D)$, a symplectic subgroup of $G$ which is the unique non-split inner form of the symplectic group $\Sp_{2n}(\K)$. A complex representation $(\pi,V)$ of $G$ is said to have a symplectic period (or to be $H$-distinguished) if there exists a linear functional $l$ on $V$ such that $l(\pi(h)v) = l(v)$ for all $v \in V$ and $h \in H$, i.e., $\Hom_{H}(\pi|_{H}, \mathbb{C}) \not=0$, where $\mathbb{C}$ denotes the trivial representation of $H$. This paper provides a complete list of irreducible admissible representations of $\GL_3(\D)$ and $\GL_4(\D)$ with a symplectic period.
	
	The importance of $H$-distinguished representations emerges prominently in the harmonic analysis of the homogeneous space $G/H$. Within the global theory of period integrals of automorphic forms, the distinguished representations play a key role in the study of special values of $L$-functions and contribute to the description of the image of functorial lifts in the Langlands program.
	
	The motivation for this problem comes from the work of Klyachko over finite fields \cite{AA84}. He discovered a family of representations $M_{n,r}$; $0 \leq r \leq [\frac{n}{2}]$,  which are disjoint and multiplicity-free, and form a complete model. When $r=0$,  the representation $M_{n,0}$ is known as the Whittaker model, which is of very much importance in number theory and the theory of automorphic representations. 
	When $n$ is even, then the representation $M_{n,\frac{n}{2}}$ is known as the symplectic model. 
	
	Heumos and Rallis \cite{HR90} studied the analogous notion of this problem over a non-archimedean local field $\K$ of characteristic zero. They established the uniqueness of the symplectic model and classified all unitary representations of $\GL_4(\K)$ possessing a symplectic model. Subsequent work by Offen and Sayag extended these classifications to unitary representations of $\GL_{2n}(\K)$, demonstrating the uniqueness of Klyachko models and proving multiplicity one results for irreducible admissible representations \cite{OSE09,OSE07,OSE08}.
	
	The study of irreducible admissible representations of $\GL_4(\K)$ and $\GL_6(\K)$ that possess a symplectic model was conducted by Mitra \cite{mitra2014representations}. Subsequently, Mitra, Offen, and Sayag \cite{mitra2017klyachko} classified ladder representations of $\GL_{2n}(\K)$ with a symplectic model. For the case where $G = \GL_n(\K)$ and $H = \GL_{n-1}(\K)$, the distinction problem has been examined by Venketasubramanian \cite{Venket}. Verma \cite{Verma} explored the problem of the symplectic period over $\D$ and demonstrated its uniqueness.
	\begin{thm}\label{1}
		Let $\pi$ be an irreducible admissible representation of $\GL_n(\D)$. Then \[ \Dim \Hom_{\Sp_n(\D)}(\pi , \mathbb{C}) \leq 1.\]
	\end{thm} 
	A natural question now is to classify all irreducible admissible representations of $\GL_n(\D)$ with a symplectic period. The second author studied this problem for $n = 1$ and $2$ in \cite{Verma}. This paper examines the problem of symplectic period for $n = 3$ and $4$, beginning with an analysis of supercuspidal representations, which are known not to exhibit a symplectic period for the supercuspidal representations of $\GL_n(\K)$.
	
    By Conjecture $7.1$ and Proposition $7.3$ in \cite{Verma}, the supercuspidal representations of $\GL_n(\D)$ for any even $n$, are not $\Sp_n(\D)$-distinguished. Therefore, for n even, it is enough to study the problem of the symplectic period for the non-supercuspidal representations of $\GL_n(\D)$. On the other hand, D. Prasad \cite{Verma} constructed some examples of supercuspidal representations of $\GL_n(\D)$ having a symplectic period for any odd $n$. So, it is interesting to study distinguished supercuspidal representation in odd cases that we will consider in the future.
	
	\subsection{Main Theorems} Our main results are summarized in the following theorems (see Section $2$ for unfamiliar notation used here).
	
	\begin{thm}\label{2}
		Let $\theta$ be an irreducible admissible non-supercuspidal representation of $\GL_3(\D)$ with a symplectic period. Then $\theta$ is one of the following:
		\begin{enumerate}
			\item[\upshape(1)] Any character of $\GL_3(\D)$.
			
			\item[\upshape(2)] $\chi_2 St_2 \times \chi_1$, where $\chi_2 St_2$ is a subrepresentation of $\chi_1\nu \times \chi_1\nu^{-1}$; $\chi_1$ is a character of $\GL_1(\D)$.
			
			\item[\upshape(3)]  $\lambda_1 \times \lambda_2$, where $\lambda_1 $ is a character of $\GL_1(\D)$ and $\lambda_2$ is an irreducible  representation of $\GL_2(\D)$ with a symplectic period.
			
		\end{enumerate}
	\end{thm}

	\begin{thm}\label{3}
		An  irreducible admissible representation $\theta$ of $\GL_4(\D)$ has a symplectic period if and only if $\theta$ is one of the following:   
		\begin{enumerate}
			
			\item[\upshape(1)]    Any character of $\GL_4(\D)$.

			\item[\upshape(2)]  $\chi_1 \times \chi_1 L([\nu^{-1} ,\nu ],[\nu^{3} ])$, where $\chi_1$ is a character of $\GL_1(\D)$.
			
			\item[\upshape(3)]  $\nu^2 \chi_1 \times \chi_1 L([\nu ,\nu^3 ],[\nu^{-1} ])$, where $\chi_1$ is a character of $\GL_1(\D)$.

			\item[\upshape(4)]   $ \chi_2St_2 \times \nu\chi_2St_2$, where $\chi_2 St_2$ is the Steinberg representation of $\GL_2(\D)$.
			
			\item[\upshape(5)]   $Z([\mu \nu^{-1/2} , \mu \nu^{5/2}])$,  where $\mu$ is a representation of $\GL_1(\D)$ with $\Dim(\mu)>1$.
			
			\item[\upshape(6)]  $Z([\pi_2 , \nu\pi_2])$, where $\pi_2$ is an irreducible supercuspidal representation of $\GL_2(\D)$.

			\item[\upshape(7)]   $Z([\mu\nu^{-3/2}, \mu\nu^{-1/2}] , [\mu\nu^{-1/2} , \mu\nu^{1/2}])$, where $\mu$ is a representation of $\GL_1(\D)$ with $\Dim(\mu)>1$.
			
			\item[\upshape(8)]  $\sigma_1 \times \sigma_2$, where $\sigma_1$ and $\sigma_2$ are irreducible representations of $\GL_2(\D)$ with a symplectic period. 
			
			\item[\upshape(9)]  $\lambda_1 \times \lambda_2$, where $\lambda_1 $ is a character of $\GL_1(\D)$ and $\lambda_2$ is an irreducible  representation of $\GL_3(\D)$ with a symplectic period. 
			
		\end{enumerate}
	\end{thm}
	
	For $n = 3~ \text{and}~  4$, Theorem  \ref{2} and Theorem \ref{3} also contain all $\Sp_n(\D)$-distinguished unitary representations of $\GL_n(\D)$ postulated in Prasad's conjecture (Conjecture $7.1$, \cite{Verma}). Note that for $n =3$, we only deal with non-supercuspidal representations having a symplectic period.
	\vspace{-.1cm}
	\subsection{About the Proofs} By Conjecture $7.1$ and Proposition $7.3$ in \cite{Verma}, the irreducible supercuspidal representations of $\GL_4(\D)$ are not $\Sp_4(\D)$-distinguished. Therefore, we focus  on studying irreducible non-supercuspidal representations with a symplectic period. Since any irreducible non-supercuspidal representation is a quotient of a representation of the form $\ind_{P_{r,n-r}}^{\GL_{n}(\D)}(\lambda_1 \otimes \lambda_2)$, $\lambda_1 \in \Irr(\GL_r(\D))$, $\lambda_2 \in \Irr(\GL_{n-r}(\D))$, it is enough to study the problem for representations of these types. For $\GL_{3}(\D)$, this reduces the problem to the analysis of representations of the type $\sigma_1 \times \sigma_2$, where $\sigma_1 \in \Irr(\GL_2(\D))$, $\sigma_2 \in \Irr(\GL_1(\D))$. For $\GL_{4}(\D)$, this breaks down the problem to the analysis of $\lambda_1 \times \lambda_2$, where either $\lambda_1 , \lambda_2 \in \Irr(\GL_2(\D))$ or $\lambda_1 \in \Irr(\GL_1(\D))$ and $\lambda_2$ is irreducible supercuspidal representation of $\GL_3(\D)$. In both instances, an exhaustive list of representations (not necessarily irreducible) is acquired using Mackey's theory. Subsequently, a study of every subquotient (up to permutation) through open and closed orbits yields a complete list of irreducible non-supercuspidal representations of $\GL_n(\D)$ with a symplectic period for $n = 3$ and $4$. In the context of an open orbit, we make use of the functor $\Ext_{\Sp_n(\D)}^1(\cdot , \mathbb{C})$.
	
	\subsection{Organization} The structure of the paper is outlined as follows. Section 2 introduces the notation and preliminary notions, focusing on the theory of segments, Zelevinsky, and Langlands classification theorems. We deal with orbit structure and compute stabilizers of the orbits in Section $3$. Sections $4$ and $5$ analyze the representations of the groups $\GL_3(\D)$ and $\GL_4(\D)$ with a symplectic period,  leading to the proofs of Theorem \ref{2} and Theorem \ref{3}. 
	
	
	\section{Preliminaries}
	\subsection{Notation}
	Let $\D$ be the unique quaternion division algebra over a non-archimedean local field $\K$ of characteristic zero. Let $\Te_{\D/\K}$ and $\N_{\D/\K}$ be the reduced trace and reduced norm maps on $\D$ and  $\tau$ be the involution on $\D$ defined by $x \rightarrow \overline{x} = \Te_{\D/\K}(x) - x$.
	
	\subsubsection{}  For $n \in \mathbb{N},$ let $V_n = e_1\D \oplus~ e_2\D \oplus \cdots \oplus e_n\D $ be a right $\D$-vector space of dimension $n$.
	\begin{defn}
		We define a Hermitian form $(,)$ on $V_n$ by 
		\begin{enumerate}
			\item $( e_i , e_{n-j+1} ) = \delta_{ij}~ for ~i,j = 1,2,\ldots,n$,
			\item $( v , v' ) = \tau( v' , v )$,
			\item $( vx , v'x' ) = \tau(x)( v , v' )x',~ for ~ v,v' \in V_n ,~ x,x' \in \D.$
		\end{enumerate}
	\end{defn}
	
	\subsubsection{} Let $\Sp_n(\D)$ be the group of isometries of the Hermitian form $(,)$. The group $\Sp_n(\D)$ is the unique non-split inner form of the group $\Sp_{2n}(\K)$. The group $\Sp_n(\D)$ can also be defined as 
	$$\Sp_n(\D) = \{ A \in \GL_n(\D)~ | ~ A ~J ~ ^t\bar{A} = J \},$$  
	where $^t\bar{A} = (\bar{a}_{ji})$ for $A = (a_{ij})$ and 
	$$ J = \begin{pmatrix}
		&  &  &    1 \\
		&  &  1  & \\
		& {\iddots}  &  & \\
		1  &  &  & 
	\end{pmatrix}.$$
	
	For a right $\D$-vector space $V,$ let $\GL_{\D}(V)$ be the group of all invertible $\D$-linear transformations on $V$ and $\Sp_{\D}(V)$ be the group of all invertible $\D$-linear transformations on $V$, which preserve the above-defined Hermitian form on $V$. The group  $\GL_{\D}(V)$ is identified with $\GL_n(\D)$  using the above basis. Since we have a Hermitian structure on $V$, so $H = \Sp_{\D}(V) \subset \GL_{\D}(V)$.
	
	\subsubsection{} In accordance with the notation used in \cite{BZ}, we denote the set of all smooth representations of an $l$-group $G$ as $\Alg(G)$ and the subset of all
	irreducible admissible representations as $\Irr(G)$.  For any $\pi \in \Alg(G)$, let $\tilde{\pi}$ be its contragredient representation. Denote the modular character of the group $G$ by $\delta_{G}$. Let $\Ext_{G}^1(\pi_1,\pi_2)$ represent the derived functor of the $\Hom_{G}(\pi_1,\pi_2)$.
	
	\subsubsection{} Let $\Irr(\GL_r(\D))$ be the set of all irreducible admissible representation of $\GL_r(\D)$. We denote $\Irr(G') = \cup_{r \geq 0} \Irr(\GL_r(\D))$. Let $C$ be the set of all supercuspidal representations in $\Irr(G')$. Let $\mathbb{C}$ denote the trivial representation of any group $K$. 
	
	\subsubsection{}  Define $\nu = \nu_n = |\N_{\D|\K}(\cdot)|_{\K}$ as a character on $\GL_n(\D)$ for any $n \in \mathbb{N}$. Note that we assume $\nu_1 = \nu^2$ throughout the paper. For all $x \in \D^{\times}$, set
	$|x|_{\D} = |\N_{\D/\K}(x)|_{\K}^2 = \nu^2(x)$.

	\subsubsection{} Let $P_{r,k-r}$ be the group of block upper triangular matrices corresponding to the tuple $(n_1,\ldots,n_r)$, with unipotent radical $N_{r,k-r}$. The modulus function of the group $P_{r,k-r}$ is denoted by $\delta_{P_{r,k-r}}$. Since a parabolic group normalizes its unipotent radical, it defines a character of $P_{r,k-r}$ through
	​the module of the automorphism $n \mapsto pnp^{-1}$ of $N_{r,k-r}$ for $p \in P_{r,k-r}$. Call this character $\delta_{N_{r,k-r}}$, thus $\delta_{N_{r,k-r}} = \delta_{P_{r,k-r}}$. For an element $p \in P_{r,k-r}$ with its Levi part equal to $\diag(g_1,g_2)$, we have 
	$$\delta_{P_{r,k-r}}(p) = |\det(g_1)|_{\D}^{k-r}|\det(g_2)|_{\D}^{-r}.$$
	
	\subsubsection{} Consider the induced representation of $(\sigma, H, W) \in \Alg(H)$ to $G$ as the space of locally constant functions 
	$$\Ind_{H}^G(\sigma) = \{f : G \rightarrow W ~|~ f(hg) = \delta_H^{1/2}\delta_G^{-1/2} \sigma(h)f(g) ~ \text{for all}~ h \in H, g\in G\},$$
	where $\delta_{G}$ and $\delta_{H}$ are the modular functions of $G$ and $H$, respectively. The group $G$
	acts on this space by right translation. The subspace of 
	$\Ind_{H}^G(\sigma)$ consisting of functions compactly supported modulo $H$ is known as the compact induction from $H$ to $G$, denoted by $\ind_H^G(\sigma)$.
    Occasionally we will use non-normalized induction (see Remark $2.22$ of \cite{BZ} for the definition), although unless otherwise mentioned induction is always normalized.  
	 
	 Given representations $\rho_i \in \Irr(\GL_{m_i})$ with $i = 1, \ldots, r$, extend $\rho_1 \otimes \cdots \otimes \rho_r$ to the parabolic subgroup $P_{m_1,\ldots,m_r}$ such that it acts trivially on the unipotent radical $N_{m_1,\ldots,m_r}$. We denote the representation $\Ind_{P_{m_1,\ldots,m_r}}^{\GL_{m_1+\cdots+m_r}}$ by $\rho_1 \times \cdots \times \rho_r$. 	
	The Jacquet functor with respect to a unipotent subgroup N is denoted by $r_N$
	and is always normalized (for more details about Jacquet functor, see Section $1$ in \cite{BZ}).


	\subsection{Preliminaries on segments and symplectic periods}
	
	\subsubsection{} Let $\rho$ be a supercuspidal representation of $\GL_r(\D)$ and $a,b \in \mathbb{Z}$ such that $a \leq b$.
	\begin{defn}\label{4}
		A segment is a sequence of the form $(\rho  \nu_{\rho}^a, \rho \nu_{\rho}^{a+1}, \ldots ,  \rho \nu_{\rho}^{b})$, where $\nu_{\rho} = \nu$ or $\nu^2$. We denote a segment by $[a,b]_{\rho}$ or $\Delta$.
	\end{defn}
	\begin{rem}
		In Definition \ref{4}, the character $\nu_{\rho}$ depends on the representation $\rho$ (for more details, see Section $2$ in  \cite{tadic1990induced}).
	\end{rem}
	For a segment $\Delta = [a,b]_{\rho}$ as above,  the unique irreducible submodule and the unique irreducible quotient of $\rho \nu_{\rho}^a \times \cdots \times \rho \nu_{\rho}^{b}$ are denoted by $Z(\Delta)$ and $L(\Delta)$, respectively. We denote by $b(\Delta) = \rho \nu_{\rho}^a$ the beginning, $e(\Delta) = \rho \nu_{\rho}^b$ the end, and $l(\Delta) = b-a+1$ the length of $\Delta$. Let $\tilde{\Delta}=[-b,-a]_{\rho}$ denote the contragredient of the segment $\Delta$.

	\begin{defn}
		Let $\Delta = [a,b]_{\rho}$ and $\Delta' = [a',b']_{\rho'}$ be two segments. We say that $\Delta$ precedes $\Delta'$ if a subsequence can be extracted from the sequence $(\rho  \nu_{\rho}^a, \ldots ,  \rho \nu_{\rho}^{b}, \ldots, \rho'\nu_{\rho'}^{b'})$, forming a segment of length strictly greater than $l(\Delta)$ and $l(\Delta')$.
	\end{defn}
	The segments $\Delta$ and $\Delta'$ are called linked if either $\Delta$ precedes $\Delta'$ or $\Delta'$ precedes $\Delta$.
	Given a multiset $\mathfrak{a} = \{\Delta_1,\ldots,\Delta_s\}$ of segments, let 
	$\lambda(\mathfrak{a}) := Z(\Delta_1) \times \cdots \times Z(\Delta_s) .$
	If $\Delta_i$ does not precede $\Delta_j$ for any $i < j$, then  $Z(\Delta_1,\ldots,\Delta_s)$ is a unique irreducible submodule of $\lambda(\mathfrak{a})$. This submodule is independent of the ordering of the segments as long as the ``does not precede condition" holds. A similar statement also holds for quotients. The unique irreducible quotient of $\pi(\mathfrak{a}) := Q(\Delta_1) \times \cdots \times Q(\Delta_s)$ is denoted by $Q(\Delta_1,\ldots,\Delta_s)$. 
	
	\subsubsection{} Let S denote the set of all segments in $C$ and  $M(S)$ indicate the set of all finite multisets in S. Then an elementary operation on $\mathfrak{a} = (\Delta_1, \Delta_2, \ldots, \Delta_n) \in M(S)$ is the replacement of two linked segments $\{\Delta_1, \Delta_2 \}$ in $\mathfrak{a}$ by the pair $\{\Delta_1 \cup \Delta_2, \Delta_1 \cap \Delta_2\}$. We put a partial order on $M(S)$ as follows: Define $\mathfrak{b} \leq \mathfrak{a}$ if $\mathfrak{b}$ can be obtained from $\mathfrak{a}$ by a sequence of elementary operations.

	\subsubsection{Zelevinsky and Langlands classification theorems} The following theorems provide the subquotients of a representation of $\GL_n(\D)$.
	
	\begin{thm}\upshape\cite{minguez2013representations}\textbf{.}\label{5}
		Let $\Delta_1$ and $\Delta_2$ be two segments. If $\Delta_1$ and $\Delta_2$ are linked, then set $\Delta_3 = \Delta_1 \cup \Delta_2 $ and $\Delta_4 = \Delta_1 \cap \Delta_2$. The representation $\pi = Z(\Delta_1) \times Z(\Delta_2)$ is irreducible if and only if $\Delta_1$ and $\Delta_2$ are not linked. If $\Delta_1$ and $\Delta_2$ are linked, then $\pi$ has length $2$. If $\Delta_2$ precedes $\Delta_1$, then $\pi$ has a unique irreducible subrepresentation $Z(\Delta_1 , \Delta_2)$ and a unique irreducible quotient $Z(\Delta_3) \times Z(\Delta_4)$. If $\Delta_1$ precedes $\Delta_2$, then $\pi$ has a unique irreducible subrepresentation $Z(\Delta_3) \times Z(\Delta_4)$ and a unique irreducible quotient $Z(\Delta_1 , \Delta_2)$.
	\end{thm}
	\begin{exa} Let $\chi_1$ be a character of $\GL_1(\D)$ and $\mu$ be an irreducible representation of $\GL_1(\D)$ with $\Dim(\mu)>1$. Let $\rho$ be an irreducible supercuspidal representation of $\GL_2(\D)$.
	\end{exa}
	\begin{enumerate}
		\item 	If $\pi = \nu_1^{-1/2}\chi_1 \times \nu_1^{1/2}\chi_1$, then it has a unique irreducible subrepresentation $\chi_2$ and a unique irreducible quotient $\chi_2 St_2$. The representation $\chi_2$ is a character, and $\chi_2 St_2$ is called the Steinberg representation.
		
		\item If $\pi = \nu^{-1/2}\mu \times \nu^{1/2}\mu$, then it has a unique irreducible subrepresentation $\Sp_2(\mu)$ and a unique irreducible quotient $St_2(\mu)$. The representations $\Sp_2(\mu)$ and $St_2(\mu)$ are called the Speh and generalized Steinberg representations, respectively.
		
		\item If $\pi = \nu^{-1/2}\rho \times \nu^{1/2}\rho$, then it has a unique irreducible subrepresentation \linebreak $Z([\nu^{-1/2}\rho , \nu^{1/2}\rho])$ and a unique irreducible quotient $Z([ \nu^{-1/2}\rho],[ \nu^{1/2}\rho])$.
	\end{enumerate}
	
	\begin{thm}\upshape\cite{minguez2013representations}\textbf{.}\label{6}
		Let $\mathfrak{a} , \mathfrak{b} \in M(S)$. Then $Z(\mathfrak{b})$ is a subquotient of $\lambda (\mathfrak{a})$ if and only if $\mathfrak{b} \leq \mathfrak{a}$.
	\end{thm}
	
	\begin{thm}\upshape\cite{tadic1990induced}\textbf{.} \label{7}
		Let $\Delta_1$ and $\Delta_2$ be two segments. The segments $\Delta_1$ and $\Delta_2$ are not linked if and only if $L(\Delta_1) \times L(\Delta_2)$ is irreducible. The segments $\Delta_1$ and $\Delta_2$ are linked if and only if $L(\Delta_1) \times L(\Delta_2)$ has two different irreducible subquotients $L(\Delta_1 \cup \Delta_2) \times L(\Delta_1 \cap \Delta_2)$ and $L(\Delta_1 , \Delta_2)$ with multiplicity one. 
	\end{thm}

	\begin{thm}\upshape\cite{tadic1990induced}\textbf{.} \label{8}
		Let $\mathfrak{a} , \mathfrak{b} \in M(S)$. Then $L(\mathfrak{b})$ is a subquotient of $\pi (\mathfrak{a})$ if and only if $\mathfrak{b} \leq \mathfrak{a}$.
	\end{thm}
	
	\subsubsection{} The following lemmas significantly contribute to the understanding of symplectic periods within $\GL_3(\D)$ and $\GL_4(\D)$ representations.
	
	\begin{lem}\label{9}
		Let $\pi$ be an admissible finite length representation of $\GL_n(\D)$.~If $\pi$ has a symplectic period, then there exists an irreducible subquotient $\tau$ of $\pi$, with a symplectic period.
	\end{lem}
	
	\begin{lem}\upshape\cite{Verma}\textbf{.}\label{10}
		Let $\pi$ be an irreducible admissible representation of $\GL_n(\D)$. If $\pi$ has a symplectic period, then its contragredient also has a symplectic period.
	\end{lem}
	\begin{lem}\upshape\cite{raghuram5541}\textbf{.}\label{11}
		Let $\pi$ and $\pi'$ be a smooth representation of $\D^{\times}$. Then 
		$$\Dim_{\mathbb{C}}\Ext_{\D^{\times}}^1(\pi , \pi') = 
		\begin{cases}
			0 & \text{if } \pi \not= \pi'\\
			1 & \text{if } \pi = \pi'
		\end{cases}
		.$$
	\end{lem}
	\subsubsection{} The following lemmas provide  the classification of irreducible representations of $\GL_1(\D)$ and $\GL_2(\D)$, with a symplectic period.
	
	\begin{lem}\upshape\cite{Verma}\textbf{.}\label{12} Any finite-dimensional irreducible representation $\sigma$ of $\GL_1(\D)$ has a symplectic period if and only if $\sigma$ is a character of $\GL_1(\D)$.
	\end{lem}
	
	\begin{lem}\upshape\cite{Verma}\textbf{.}\label{13}
		Let $\sigma$ be an irreducible admissible representation of $\GL_2(\D)$ with a symplectic period. Then $\sigma$ is one of the following: 
		\begin{enumerate}
			\item[\upshape(1)] Any character of $\GL_2(\D)$.
			
			\item[\upshape(2)] $\sigma_1 \times \sigma_2$, where $\sigma_1$ and $\sigma_2$ are characters of $\GL_1(\D)$. 
			
			\item[\upshape(3)] Speh representation $\Sp_2(\mu)$, where $\mu$ is a representation of $\GL_1(\D)$ with $\Dim(\mu)>1$.
			
		\end{enumerate}
	\end{lem}
	\section{Orbit Structures and Stabilizers}
	This section focuses on the double cosets  $H \backslash G / P$ for  $G = \GL_n(\D)$, $H = \Sp_n(\D)$, and  maximal parabolic subgroup $P = P_{k,n-k}$ of $\GL_n(\D)$, where $n=3,4$.
	\subsection{$\Sp_3(\D)$-orbits in $\GL_3(\D)/P_{2,1}$}  We first explore the structure of $\Sp_3(\D) \backslash  \GL_3(\D)/P_{2,1}$.
	\subsubsection{} Let $V$ be a $3$-dimensional Hermitian right $\D$-vector space with a basis $\{e_1,e_2,e_3\}$ of $V$ with $(e_1,e_1) = (e_1,e_2) = (e_2,e_3) = (e_3,e_3) = 0 $ and $(e_1,e_3) = (e_2,e_2) = 1$. Let $X$ be the set of all $2$-dimensional $\D$-subspaces of $V$. The group $G = \GL_{\D}(V)$ acts naturally on $V$ and induces a transitive action on $X$. Hence, the stabilizer in $\GL_3(\D)$ of the plane $W = \langle (e_1,e_2) \rangle$ inside $X$ is a parabolic subgroup $P_{2,1}$ in $\GL_3(\D)$, with $X \simeq G / P$. 
	
	\subsubsection{Orbits and stabilizers}\label{14} Our aim is to study the space $\Sp_3(\D) \backslash \GL_3(\D)/ P_{2,1}$. Since the flag variety $X$ is in bijection with $\GL_3(\D)/ P_{2,1}$ and by Witt's theorem  {\cite{MVW}} the action of $\Sp_3(\D)$ on maximal isotropic subspaces $W \in X$ is transitive, therefore, this leads to the formation of two distinct orbits, namely $O_1 = \{W \in X ~|~ \Dim \Rad J\mid_{X} = 1 \}$ and $O_2 = \{W \in X ~|~ \Dim \Rad J\mid_{X} = 0 \}$.
	
	 Consider the plane $\langle (e_1,e_2) \rangle$ inside $O_1$. Then the stabilizer in $\Sp_3(\D)$ of the plane $\langle (e_1,e_2) \rangle$  is the parabolic subgroup
	\begin{multline*}
		P_{H} = \Bigg\{
		\begin{pmatrix}
			a & b & c \\
			0 & e & f \\ 
			0 & 0 & \overline{a}^{-1} 
		\end{pmatrix} \Bigg{|}~ a,e \in \D^{\times};~b,c,f \in \D; ~ e\overline{e} =1 ,~ a\overline{f}+b\overline{e}=0,~ a\overline{c}+c\overline{a}+b\overline{b}=0 \Bigg\}.
	\end{multline*}
	 of $\Sp_3(\D)$ with Levi decomposition $P_{H} = M_{H} U_{H}$, where 
	\begin{equation*}
		M_{H} = \Bigg{\{}\begin{pmatrix}
			a & 0 & 0 \\
			0 & e & 0 \\ 
			
			0 & 0 & \overline{a}^{-1} 
		\end{pmatrix} \Bigg{|} a \in \GL_1(\D),~e \in \Sp_1(\D) \Bigg{\}} \simeq \Delta(\GL_1(\D) \times \GL_1(\D)) \times \Sp_1(\D).
	\end{equation*}
	
	Now, consider the plane $\langle (e_1 + e_3, ~ e_1 - e_3) \rangle $ inside $O_2$. If an isometry of $V$ stabilizes the plane $\langle (e_1 + e_3, ~ e_1 - e_3) \rangle $ in $\Sp_3(\D)$, then it also stabilizes their orthogonal complement $\langle e_2 \rangle $. Hence, the stabilizer of the plane $\langle (e_1 + e_3, ~ e_1 - e_3) \rangle $ in $\Sp_3(\D)$ stabilizes the orthogonal decomposition of $V$ as
	$V = \langle (e_1 + e_3, ~ e_1 - e_3) \rangle  \oplus \langle e_2 \rangle $.
	Thus, the stabilizer in $\Sp_3(\D)$ of the plane $\langle (e_1 + e_3, ~ e_1 - e_3) \rangle $ is $ \Sp_2(\D) \times \Sp_1(\D)$ sitting in a natural way in the Levi $\GL_2(\D) \times \GL_1(\D)$ of the parabolic $P_{2,1}$ in $\GL_3(\D)$.
	
	\subsection{$\Sp_4(\D)$-orbits in $\GL_4(\D)/ P_{k,4-k}$} Now, we delve into the systematic analysis of $\Sp_4(\D)$-orbits in  $P_{k,4-k} \backslash \GL_4(\D)$, where $k = 1,2$.
	
	\subsubsection{}\label{55} Let $V$ be a $4$-dimensional Hermitian right $\D$-vector space with a basis $\{e_1,e_2,e_3,e_4\}$ of $V$ with $(e_i,e_{5-i}) = 1$ otherwise $0$. Let $X$ be the set of all $k$-dimensional $\D$-subspaces of $V$. Since the group $\GL_4(\D)$ acts transitively on $X$, therefore, the stabilizer in $\GL_4(\D)$ of the line $W = \langle e_1,\ldots,e_k \rangle $ inside $X$  is a parabolic subgroup $P_{k,4-k}$, with $X \simeq \GL_4(\D)/ P_{k,4-k}$.
	
	\subsubsection{Orbits and stabilizers}\label{16}
	The transitive nature of $\Sp_4(\D)$ on maximal isotropic subspaces implies the existence of $(k+1)$-distinct orbits, denoted as $O_r =  \{ W \in X ~|~ \Dim \Rad J\mid_{X} = r \}$ where $0 \leq r \leq k$. Consider $T_k = \langle e_1,\ldots,e_k\rangle$ inside $O_k$ and $$T_r = \langle(e_1, \ldots, e_r,e_{r+1}+ e_{n-r},\ldots, e_k+e_{n-k+1})\rangle$$ inside $O_r$, where $0 \leq r \leq (k-1)$. Then we have the following proposition.
	\begin{prop}\label{15}
		The subgroup $H_{k,r}$ of $\Sp_4(\D)$ stabilizing the subspaces $T_{r}$ of $V$ is of the form 
		$$M_{k,r} U_{k,r} = (\Delta (\GL_r(\D) \times \GL_r(\D)) \times \Sp_{k-r}(\D) \times \Sp_{4-k-r}(\D))\cdot U_{k,r},$$
		where $0 \leq r \leq k$, $\Delta (\GL_r(\D) \times \GL_r(\D)) \simeq \{(g ,~ ^t\overline{g}^{-1}) ~|~ g \in \GL_r(\D)\}$, and $U_{k,r}$ is the unipotent subgroup inside $\Sp_4(\D)$. 
	\end{prop}

	\section{Distinguished representations of $\GL_3(\D)$}
	
	This section analyzes the non-supercuspidal irreducible admissible representations of $\GL_3(\D)$ with a symplectic period and proves Theorem \ref{2}.
	
	\subsection{Necessary and sufficient conditions on distinction for a  representation of $\GL_3(\D)$}  
	\subsubsection{} We begin with the following lemma. 
	\begin{lem}\label{20}
		Let $\theta$ be a non-supercuspidal irreducible admissible representation of $\GL_3(\D)$ with a symplectic period. Then $\theta$ appears as a quotient of $\sigma_1 \times \sigma_2$, where $\sigma_1 \in \Irr(\GL_2(\D))$ and $\sigma_2 \in \Irr(\GL_1(\D))$.
	\end{lem}
	\begin{proof}
		If $\theta$ is a non-supercuspidal irreducible admissible representation of $\GL_3(\D)$, then $\theta$ appears as a quotient of either $\sigma_1 \times \sigma_2$ or $\sigma_2 \times \sigma_1$, where $\sigma_1 \in \Irr(\GL_2(\D))$ and $\sigma_2 \in \Irr(\GL_1(\D))$. The first case is trivial. If $\theta$ is a quotient of $\sigma_2 \times \sigma_1$, then $\tilde{\theta}$ is  a quotient of $\tilde{\sigma_1} \times \tilde{\sigma_2}  $. By Lemma \ref{10}, we are reduced to the first case. 
	\end{proof}

	\subsubsection{Mackey theory}\label{21} By Lemma \ref{20}, we apply the Mackey theory to the representations of the form $\sigma_1 \times \sigma_2$, where $\sigma_1 \in \Irr(\GL_2(\D))$ and $\sigma_2 \in \Irr(\GL_1(\D))$. 
	
	Given the representation $\pi = \sigma_1 \times \sigma_2 := \Ind_{P_{2,1}}^{\GL_3(\D)}\sigma$ of $\GL_3(\D)$, where $\sigma = \sigma_1 \otimes \sigma_2$ is an irreducible representation of $\GL_2(\D) \times \GL_1(\D)$, consider the restriction of $\pi$ to $\Sp_3(\D)$.
	 By applying Mackey theory and $\S$\ref{14}, we obtain the following exact sequence of $\Sp_3(\D)$-representations 
	$$0\rightarrow \ind_{\Sp_2(\D) \times \Sp_1(\D)}^{\Sp_3(\D)}[(\sigma_1 \otimes \sigma_2)|_{\Sp_2(\D) \times \Sp_1(\D)}] \rightarrow \pi \rightarrow  \Ind_{P_{H}}^{\Sp_3(\D)} \delta_{P_{H}}^{1/2}[(\sigma_1 \otimes \sigma_2)|_{M_{H}}] \rightarrow 0,$$
	where $\delta_{P_{H}}$ is the character on $P_{H}$ given by 
	\begin{equation*}
		\delta_{P_{H}} \Bigg{[}\begin{pmatrix}
			a & b & c \\
			0 & e & f \\ 
			
			0 & 0 & \overline{a}^{-1} 
		\end{pmatrix}\Bigg{]} = |\N_{\D/\K}(a)|_{\K}^{6}.
	\end{equation*}
	
	\subsubsection{Analysis by the open and closed orbit}\label{23}
	Suppose $\pi$ has a non-zero $\Sp_3(\D)$-invariant linear form. Then one of the representations in the above exact sequence,
	$$\ind_{\Sp_2(\D) \times \Sp_1(\D)}^{\Sp_3(\D)}[(\sigma_1 \otimes \sigma_2)|_{\Sp_2(\D) \times \Sp_1(\D)}]~\text{or}~\Ind_{P_{H}}^{\Sp_3(\D)} \nu^{3}[(\sigma_1 \otimes \sigma_2)|_{M_{H}}],$$
	must have an $\Sp_3(\D)$-invariant form.
	 Consider the case when 
	$$\Hom_{\Sp_3(D)}(\ind_{\Sp_2(\D) \times \Sp_1(\D)}^{\Sp_3(\D)}[(\sigma_1 \otimes \sigma_2)|_{\Sp_2(\D) \times \Sp_1(\D)}],\mathbb{C}) \not= 0.$$
	Then by Frobenius reciprocity, it is equivalent to 
	$$\Hom_{\Sp_2(\D) \times \Sp_1(\D)}((\sigma_1 \otimes \sigma_2),\mathbb{C}) \not= 0,$$
	and hence
	$$\Hom_{\Sp_2(\D)}(\sigma_1,\mathbb{C})~\otimes~ \Hom_{\Sp_1(\D)}(\sigma_2,\mathbb{C}) \not= 0.$$
	It follows that the representations $\sigma_1$ and $\sigma_2$ have a symplectic period. 
	
	 Now, assume that $$\Hom_{\Sp_3(\D)}(\Ind_{P_{H}}^{\Sp_3(\D)} \nu^{3}[(\sigma_1 \otimes \sigma_2)|_{M_{H}}],\mathbb{C}) \not= 0.$$ 
	Then Frobenius reciprocity implies that the latter space is non-zero if and only if
	\begin{equation}\label{57}
	\Hom_{P_{H}}(\nu^{3}[(\sigma_1 \otimes \sigma_2)|_{M_{H}}] ~,  \nu^5) \not= 0.
	\end{equation}
	Clearly, 
	$$\Hom_{M_{H}U_{H}}(\nu^{-2}[(\sigma_1 \otimes \sigma_2)|_{M_{H}}] ~,  \mathbb{C}) = \Hom_{M_{H}N}(\nu^{-2}[(\sigma_1 \otimes \sigma_2)|_{M_{H}}] ~, \mathbb{C}),$$
	where $N = N_1 \times N_2$; $N_1 , N_2$ are the unipotent subgroups of $\GL_2(\D)$ and $\GL_1(\D)$ corresponding to the partition $(1,1)$ and $(0,1)$, respectively. With the normalized Jacquet functor being left adjoint to normalized induction (Proposition $1.9$(b) in \cite{BZ}), we obtain 
	\begin{equation}\label{58}
	\Hom_{M_{H}N}(\nu^{-2}[(\sigma_1 \otimes \sigma_2)|_{M_{H}}] ~, \mathbb{C}) = \Hom_{M_{H}}(r_N(\nu^{-2}\sigma_1 \otimes \sigma_2)~ , \delta_N^{-1/2}).
	\end{equation}
 The right-hand side of $(\ref{58})$ is equivalent to $\Hom_{M_{H}}(\nu^{-2} \delta_{N_1}^{1/2}r_{N_1}(\sigma_1) \otimes \sigma_2 ~, \mathbb{C})$.
	Since $e \in \Sp_1(\D)$, it is easy to see that
	$$\delta_{N_1} \begin{pmatrix}
		a & 0  \\
		0 & e   
	\end{pmatrix} = |\N_{\D|\K}(a)|_{\K}.$$
	Therefore, it is easy to observe that the space in $(\ref{57})$ is non-zero if and only if 
	$$\Hom_{\Delta(\GL_1(\D) \times \GL_1(\D)) \times \Sp_1(\D)}(\nu^{-1}r_{(1,1),(2)}(\sigma_1) \otimes \sigma_2~,\mathbb{C}) \not= 0,$$
	where $\GL_1(\D)$ acts on $\sigma_2$ (last representation) via the contragredient. 
	\subsubsection{}
	From the analysis in $\S$\ref{23}, we deduce the following propositions 
	which provide necessary and sufficient conditions on distinction for a parabolically induced representation of $\GL_3(\D)$.
	\begin{prop} \label{25}
		If the representation 
		$$\pi = \sigma_1 \times \sigma_2 := \Ind_{P_{2,1}}^{\GL_3(\D)}(\sigma_1 \otimes \sigma_2)$$ has an $\Sp_3(\D)$-invariant linear form, then either
		\begin{enumerate}
			\item[\upshape(1)]  $\Hom_{\Delta(\GL_1(\D) \times \GL_1(\D)) \times \Sp_1(\D)}(\nu^{-1}r_{(1,1),(2)}(\sigma_1) \otimes \sigma_2~,\mathbb{C}) \not= 0$~(\text{closed orbit condition}) or 
			
			\item[\upshape(2)]   $\sigma_1$ and $\sigma_2$ have a symplectic period~(\text{open orbit condition}).
		\end{enumerate}
	\end{prop}
	
 
	\begin{prop}\label{26}
		The representation 
		$$\pi = \sigma_1 \times \sigma_2 := \Ind_{P_{2,1}}^{\GL_3(\D)}(\sigma_1 \otimes \sigma_2)$$ has an $\Sp_3(\D)$-invariant linear form if one of the following holds
		\begin{enumerate}
			\item[\upshape(1)] $\Hom_{\Delta(\GL_1(\D) \times \GL_1(\D)) \times \Sp_1(\D)}(\nu^{-1}r_{(1,1),(2)}(\sigma_1) \otimes \sigma_2~,\mathbb{C}) \not= 0$,
			\item[\upshape(2)]    $\sigma_1$ and $\sigma_2$ have a symplectic period with $$\Ext_{\Sp_3(D)}^1(\Ind_{P_{H}}^{\Sp_3(D)} \nu^{3}[(\sigma_1 \otimes \sigma_2)|_{M_{H}}] , \mathbb{C}) = 0.$$
		\end{enumerate} 
	\end{prop}
	
	\subsubsection{} We now analyze the functor $$\mathcal{P}_1 = \Ext_{\Sp_3(\D)}^1(\Ind_{P_{H}}^{\Sp_3(\D)} \nu^{3}[(\sigma_1 \otimes \sigma_2)|_{M_{H}}] , \mathbb{C})$$
	mentioned in the Proposition $\ref{26}$. Recall that $\sigma_1 \in \Irr(\GL_2(\D))$, $\sigma_2 \in \Irr(\GL_1(\D))$, and 
	$$ P_{H} = M_{H} U_{H} = (\Delta(\GL_1(\D) \times \GL_1(\D)) \times \Sp_1(\D))U_{H}~ (\text{see}~ \S~ \ref{14}).$$
	The following notations will be used in the context of the next theorem.
	Assume $r_{(1,1),(2)}(\sigma_1) = \sum_{i=1}^{t}(\sigma_{1i} \otimes \sigma_{1i}')$ and define 
	$$\mathcal{Q}_i = \Hom_{\GL_1(\D)}(\sigma_{1i} ,\sigma_2  \otimes \nu ) \otimes \Ext_{\Sp_2(\D)}^1(\sigma_{1i}' , \mathbb{C})$$ and $$\mathcal{R}_i = \Ext_{\GL_1(\D)}^1(\sigma_{1i} ,\sigma_2 \otimes \nu) \otimes \Hom_{\Sp_2(\D)}(\sigma_{1i}' , \mathbb{C}),$$
	for each $i = 1,2,\ldots,t$. Then we have the following theorem related to the functor $\mathcal{P}_1$. 
	\begin{thm} \label{27}
		The functor $\mathcal{P}_1$ can be expressed as a product of the direct sum of the functors $\mathcal{Q}_i$ and $\mathcal{R}_i$, i.e., 
		$$\mathcal{P}_1 = \prod_{i=1}^t (\mathcal{Q}_i \oplus \mathcal{R}_i).$$
	\end{thm}
	\begin{proof}
		By Frobenius reciprocity and  Proposition $2.5$ in \cite{prasad2013ext}, we get 
		$$\mathcal{P}_1 = \Ext_{\Sp_3(\D)}^1(\Ind_{P_{H}}^{\Sp_3(\D)} \nu^{3}[(\sigma_1 \otimes \sigma_2)|_{M_{H}}] , \mathbb{C}) = \Ext_{M_{H}U_{H}}^1(\nu^{3}((\sigma_1 \otimes \sigma_2)|_{M_{H}}) ~,  \nu^5).$$
		Clearly, $$\Ext_{M_{H}U_{H}}^1(\nu^{-2}((\sigma_1 \otimes \sigma_2)|_{M_{H}}) ~,  \mathbb{C}) = \Ext_{M_{H}N}^1(\nu^{-2}((\sigma_1 \otimes \sigma_2)|_{M_{H}}) , \mathbb{C}).$$
		Since the normalized Jacquet functor is left adjoint to normalized induction (by Proposition $2.7$ in \cite{prasad2013ext}), we have
		\begin{equation} \label{56}
			\Ext_{M_{H}N}^1(\nu^{-2}((\sigma_1 \otimes \sigma_2)|_{M_{H}}) , \mathbb{C}) = \Ext_{M_{H}}^1(r_N(\nu^{-2}\sigma_1 \otimes \sigma_2)~ , \delta_N^{-1/2}).
		\end{equation}
	The space on the right hand side of $(\ref{56})$ equals $\Ext_{M_H}^1(\nu^{-1}r_{(1,1),(2)}(\sigma_1) \otimes \sigma_2~ , \mathbb{C})$.
	Since $r_{(1,1),(2)}(\sigma_1) = \sum_{i=1}^{t}(\sigma_{1i} \otimes \sigma_{1i}')$, it is equivalent to $\Ext_{M_{H}}^1(\oplus_{i=1}^t( \sigma_{1i} \otimes \sigma_{1i}' \otimes \sigma_2),\nu)$. It follows that 
	$$\Ext_{\Sp_3(\D)}^1(\Ind_{P_{H}}^{\Sp_3(\D)} \nu^{3}[(\sigma_1 \otimes \sigma_2)|_{M_{H}}] , \mathbb{C}) = \prod_{i=1}^t (\mathcal{Q}_i~\oplus~\mathcal{R}_i).$$
	\end{proof}
	
	\subsection{Representations of $\GL_3(\D)$ with a symplectic period}
	\subsubsection{}\label{28} By Lemma \ref{20}, any irreducible non-supercupidal $\Sp_3(\D)$-distinguished representation occurs as a quotient of $\pi = \sigma_1 \times \sigma_2$, where $\sigma_1 \in \Irr(\GL_2(\D))$ and $\sigma_2 \in \Irr(\GL_1(\D))$. So, it is enough to check $\Sp_3(\D)$-distinction in all irreducible quotients or subquotients (up to permutation) of the representation $\pi$. By Proposition \ref{25}, it is easy to see that the necessary conditions for $\pi$ to have a symplectic period are the following: 
	\begin{enumerate}
		\item either $\Hom_{\Delta(\GL_1(\D) \times \GL_1(\D)) \times \Sp_1(\D)}(\nu^{-1}r_{(1,1),(2)}(\sigma_1) \otimes \sigma_2~,\mathbb{C}) \not= 0$ ~or 
		
		\item  $\sigma_1$ and $\sigma_2$ have a  symplectic period. 
	\end{enumerate}
	Since any irreducible representation of $\GL_2(\D)$ is either a supercuspidal, generalized Steinberg, Speh, a character,  a twist of the Steinberg, or an irreducible principal series representation, we have six cases depending on the irreducible representation $\sigma_1$ of $\GL_2(\D)$.\\
	\textbf{Case 1:} If $\sigma_1$ is a supercuspidal representation of $\GL_2(\D)$, then  
	$$\Hom_{\Delta(\GL_1(\D) \times \GL_1(\D)) \times \Sp_1(\D)}(\nu^{-1}r_{(1,1),(2)}(\sigma_1) \otimes \sigma_2~,\mathbb{C}) = 0.$$
	As the supercuspidal representations of $\GL_2(\D)$ are not $\Sp_2(\D)$-distinguished by Lemma \ref{13}, so neither orbit contributes to an $\Sp_3(\D)$-invariant linear form. Thus, by Proposition \ref{25}, $\pi$ does not have a symplectic period.\\ 
	\textbf{Case 2:} Let $\sigma_1$ be the generalized Steinberg representation $St_2(\mu)$ of $\GL_2(\D)$, where $St_2(\mu)$ is quotient  of  $\nu^{-1/2}\mu \times \nu^{1/2}\mu$; $\mu$ is an irreducible representation of $\GL_1(\D)$ with $\Dim(\mu)>1$. Suppose 
	$$\Hom_{\Delta(\GL_1(\D) \times \GL_1(\D)) \times \Sp_1(\D)}(\nu^{-1}r_{(1,1),(2)}(\sigma_1) \otimes \sigma_2),\mathbb{C}) \not= 0,$$
	and hence
	$$\Hom_{\Delta(\GL_1(\D) \times \GL_1(\D)) \times \Sp_1(\D)}(\nu^{-1/2}\mu \otimes \nu^{-3/2}\mu \otimes \sigma_2,\mathbb{C}) \not= 0.$$
	The latter space is non-zero if and only if 
	$$\Hom_{\Delta(\GL_1(\D) \times \GL_1(\D))}(\nu^{-1/2}\mu \otimes \widetilde{\sigma_2} , \mathbb{C}) \otimes \Hom_{\Sp_1(\D)}(\nu^{-3/2}\mu,\mathbb{C}) \not= 0,$$
	and it is equivalent to 
	$$\Hom_{\Delta(\GL_1(\D) \times \GL_1(\D))}(\mu  , \sigma_2 \otimes \nu^{1/2}) \otimes \Hom_{\Sp_1(\D)}(\nu^{-3/2}\mu,\mathbb{C}) \not= 0.$$
	It follows that $\mu \simeq \sigma_2 \otimes \nu^{1/2}$ and $\mu$ is $\Sp_1(\D)$-distinguished i.e., $\Dim(\mu)=1$. Which is false. Hence,
	$$\Hom_{\Delta(\GL_1(\D) \times \GL_1(\D)) \times \Sp_1(\D)}(\nu^{-1}r_{(1,1),(2)}(\sigma_1) \otimes \sigma_2),\mathbb{C}) = 0.$$
	Since by Lemma \ref{13} the generalized Steinberg representations $St_2(\mu)$ of $\GL_2(\D)$ are not $\Sp_2(\D)$-distinguished, the participation of neither orbit occurs in an $\Sp_3(\D)$-invariant linear form. As a result, $\pi$ is not $\Sp_3(\D)$-distinguished by Proposition \ref{25}.\\
	\textbf{Case 3:} Let $\sigma_1$ be the Speh representation $\Sp_2(\mu)$ of $\GL_2(\D)$, where  $\Sp_2(\mu)$ is quotient of $\nu^{1/2}\mu \times \nu^{-1/2}\mu$; $\mu$ is an irreducible representation of $\GL_1(\D)$ with $\Dim\mu > 1$. Suppose  
	$$\Hom_{\Delta(\GL_1(\D) \times \GL_1(\D)) \times \Sp_1(\D)}((\nu^{-1}r_{(1,1),(2)}(\sigma_1) \otimes \sigma_2),\mathbb{C}) \not= 0,$$
	and it is equivalent to $$\Hom_{\Delta(\GL_1(\D) \times \GL_1(\D)) \times \Sp_1(\D)}(\nu^{-3/2}\mu \otimes \nu^{-1/2}\mu \otimes \sigma_2,\mathbb{C}) \not= 0.$$
	The latter space is non-zero if and only if 
	$\mu \simeq \sigma_2 \otimes \nu^{3/2}$ and $\mu$ is $\Sp_1(\D)$-distinguished. Which is false. Hence, 
	the part of $\pi$ supported on the closed orbit does not contribute to an $\Sp_3(\D)$-invariant linear form. As the Speh representation $\Sp_2(\mu)$ of $\GL_2(\D)$ is $\Sp_2(\D)$-distinguished, we have two sub-cases depending on $\sigma_2$:
	
	(1) If $\Dim(\sigma_2)>1$, then the open orbit also does not participate. Therefore, $\pi$ is not $\Sp_3(\D)$-distinguished in this case.
	
	(2) If $\Dim(\sigma_2)=1$, then the part of $\pi$ supported on the open orbit, a submodule of $\pi$, participates in an $\Sp_3(\D)$-invariant linear form. By Theorem \ref{27},  $$\Ext_{\Sp_3(\D)}^1(\Ind_{P_{H}}^{\Sp_3(\D)} \nu^{3}[(\sigma_1 \otimes \sigma_2)|_{M_{H}}] , \mathbb{C}) = 0.$$
	Hence,  Proposition \ref{26} implies that the irreducible representation $\pi = \Sp_2(\mu) \times \sigma_2$ is  $\Sp_3(\D)$-distinguished.\\
	\textbf{Case 4:} Let $\sigma_1$ be a character $\chi_2$ of $\GL_2(\D)$, where $\chi_2$ is a quotient of $\chi_1\nu \times \chi_1\nu^{-1}$; $\chi_1$ is a character of $\GL_1(\D)$. Suppose
	$$\Hom_{\Delta(\GL_1(\D) \times \GL_1(\D)) \times \Sp_1(\D)}((\nu^{-1}r_{(1,1),(2)}(\sigma_1) \otimes \sigma_2),\mathbb{C}) \not= 0,$$
	and it is non-zero if and only if 
	$$\Hom_{\Delta(\GL_1(\D) \times \GL_1(\D)) \times \Sp_1(\D)}(\chi_1\nu^{-2} \otimes \chi_1 \otimes \sigma_2,\mathbb{C}) \not= 0.$$
	It suggests that $\chi_1 \simeq \sigma_2  \otimes \nu^{2}$. 
	 As a result, two sub-cases arise: 
	
	(1) If $\Dim(\sigma_2)>1$, then both the orbits do not contribute to an $\Sp_3(\D)$-invariant linear form. Thus, $\pi$ fails to exhibit $\Sp_3(\D)$-distinction in this case.
	
	(2) If $\Dim(\sigma_2)=1$ and $\chi_1 \simeq \sigma_2 \otimes \nu^{2}$, then part of $\pi$ supported on the closed orbit  also participates in this case. So, the irreducible representation $\pi$ is $\Sp_3(\D)$-distinguished. If $\Dim(\sigma_2)=1$ and $\chi_1 \not\simeq \sigma_2 \otimes \nu^{2}$, then the closed orbit does not participate, and by Theorem  \ref{27} and Lemma \ref{11},
	$$\Ext_{\Sp_3(\D)}^1(\Ind_{P_{H}}^{\Sp_3(\D)} \nu^{3}[(\sigma_1 \otimes \sigma_2)]|_{M_{H}} , \mathbb{C}) = 0.$$ 
	Hence, by Proposition \ref{26} and Lemma \ref{13},  $\pi = \sigma_1 \times \sigma_2$ has a symplectic period arising from the open orbit.

	In latter case i.e., $\Dim(\sigma_2)=1$ and $\chi_1 \not\simeq \sigma_2 \otimes \nu^{2}$, Lemma \ref{5} implies that the representation $\pi = \sigma_1 \times \sigma_2$ is reducible in two cases. So, we check all irreducible subquotients of $\pi$, having a symplectic period.
	If $\pi$ is irreducible, then it is  $\Sp_3(\D)$-distinguished. Otherwise, assume $\sigma_2 \simeq \chi_1 \nu_1^{3/2}$ without loss of generality. Then by Theorem \ref{5}, the representation $\pi = \chi_2 \times \chi_1\nu^3 $ has a unique irreducible subrepresentation  $\theta_1 = Z([\chi_1\nu_1^{-1/2}, \chi_1\nu_1^{3/2} ])$ and a unique irreducible quotient $\theta_2 = Z([\chi_1\nu_1^{-1/2}, \chi_1\nu_1^{1/2}] ,[\chi_1\nu_1^{3/2}])$. The representation $\theta_1$ is a unitary character of $\GL_3(\D)$ and has a symplectic period. By Proposition \ref{25}, $\pi$ has a symplectic period and $\Dim~\Hom_{\Sp_3(\D)}(\pi|_{\Sp_3(\D)} , \mathbb{C}) = 1$. It implies that the  representation $\theta_2$ has no  symplectic period.\\
	\textbf{Case 5:} Let $\sigma_1$ be the Steinberg representation $\chi_2St_2$ of $\GL_2(\D)$, where $\chi_2St_2$ is a quotient of $\chi_1\nu^{-1} \times \chi_1\nu$; $\chi_1$ is a character of $\GL_1(\D)$. Suppose
	$$\Hom_{\Delta(\GL_1(\D) \times \GL_1(\D)) \times \Sp_1(\D)}((\nu^{-1}r_{(1,1),(2)}(\sigma_1) \otimes \sigma_2),\mathbb{C}) \not= 0,$$
	and it is equivalent to  $$\Hom_{\Delta(\GL_1(\D) \times \GL_1(\D)) \times \Sp_1(\D)}(\chi_1 \otimes \chi_1\nu^{-2} \otimes \sigma_2,\mathbb{C}) \not= 0.$$
	It follows that $\chi_1 \simeq \sigma_2$, i.e., the part of $\pi$ supported on the closed orbit contributes to an $\Sp_3(\D)$-invariant linear form. Thus, the irreducible representation  $\pi = \chi_2 St_2 \times \chi_1$ is $\Sp_3(\D)$-distinguished. If $\chi_1 \not\simeq \sigma_2$, then the closed orbit does not contribute.  By Lemma \ref{13}, the open orbit also does not contribute.  Therefore, $\pi$ is not $\Sp_3(\D)$-distinguished in this case.\\
	\textbf{Case 6:} Let $\sigma_1$ be an irreducible principal series representation $\sigma_1 = \sigma_1' \times \sigma_1'' := \Ind_{P_{1,1}}^{\GL_2(\D)}(\sigma_1' \otimes \sigma_1'')$ of $\GL_2(\D)$, where $\sigma_1' \otimes \sigma_1''$ is an irreducible representation of $\GL_1(\D) \times \GL_1(\D)$. Suppose 
	$$\Hom_{M_{H}}((\nu^{-1}r_{(1,1),(2)}(\sigma_1) \otimes \sigma_2),\mathbb{C}) \not= 0,$$
	and hence $$\Hom_{M_{H}}((\nu^{-1}(\sigma_1' \otimes \sigma_1'' \oplus \sigma_1'' \otimes \sigma_1') \otimes \sigma_2),\mathbb{C}) \not= 0.$$
	The latter space is non-zero if and only if 
	$$\Hom_{M_{H}}((\nu^{-1}\sigma_1' \otimes \nu^{-1} \sigma_1'' \otimes \sigma_2),\mathbb{C})~\oplus~ \Hom_{M_{H}}((\nu^{-1}\sigma_1'' \otimes \nu^{-1} \sigma_1' \otimes \sigma_2),\mathbb{C}) \not= 0,$$
	and it is equivalent to saying that either 
	$$\Hom_{\Delta(\GL_1(\D) \times \GL_1(\D))}(\sigma_1' , \sigma_2 \otimes \nu) \otimes \Hom_{\Sp_1(\D)}(\nu^{-1}\sigma_1'',\mathbb{C}) \not= 0$$
	\hspace{8cm} or  
	$$\Hom_{\Delta(\GL_1(\D) \times \GL_1(\D))}(\sigma_1'' , \sigma_2 \otimes \nu) \otimes \Hom_{\Sp_1(\D)}(\nu^{-1}\sigma_1',\mathbb{C}) \not= 0,\\$$
	 i.e.,
	either $\sigma_1' \simeq \sigma_2 \otimes \nu$ and $\sigma_1''$ is $\Sp_1(\D)$-distinguished representation of $\GL_1(\D)$ or $\sigma_1'' \simeq \sigma_2 \otimes \nu$ and $\sigma_1'$ is $\Sp_1(\D)$-distinguished representation of $\GL_1(\D)$.
	By Lemma \ref{13}, the irreducible principal series representation $\sigma_1 = \sigma_1' \times \sigma_1''$    is $\Sp_2(\D)$-distinguished if and only if $\Dim(\sigma_1') = \Dim(\sigma_1'') = 1$ and $\sigma_1' \not= \sigma_1'' \otimes \nu^{\pm 2}$.
	As a result, three sub-cases emerge:
	
	(1) If $\Dim(\sigma_1')>1$ and $\Dim(\sigma_1'')>1$, then both the orbits do not participate to an $\Sp_3(\D)$-invariant linear form by Lemma \ref{12} and \ref{13}. It implies that $\pi$ does not have a symplectic period.
	
	(2) If $\Dim(\sigma_1')=1$ and $\Dim(\sigma_1'')>1$, then  the open orbit does not participate. If  $\sigma_1'' \simeq \sigma_2 \otimes \nu$, then the closed orbit  participates. It follows that $\pi = \sigma_1' \times \sigma_1'' \times \sigma_2$ is $\Sp_3(\D)$-distinguished if and only if $\sigma_1'' \simeq \sigma_2 \otimes \nu$.
	
	Assume $\sigma_1' = \chi$, $\sigma_1'' = \nu^{1/2}\mu$ and $\sigma_2 = \nu^{-1/2}\mu$. 
	Then by above condition, $\pi_1 = \chi \times \nu^{-1/2}\mu \times \nu^{1/2}\mu$ is not $\Sp_3(\D)$-distinguished and $\pi_2 = \chi \times \nu^{1/2}\mu \times \nu^{-1/2}\mu$ is $\Sp_3(\D)$-distinguished. By the exact sequence of $\GL_3(\D)$-representations 
	$$0 \rightarrow \chi \times \Sp_2(\mu) \rightarrow \pi_1 \rightarrow  \chi \times St_2(\mu) \rightarrow 0,$$ 
	we conclude that the quotient $\chi \times St_2(\mu)$ is not $\Sp_3(\D)$-distinguished.
	The representation $\pi_2$ also fits in the exact sequence of $\GL_3(\D)$-representations
	$$0 \rightarrow \chi \times St_2(\mu) \rightarrow \pi_2  \rightarrow  \chi \times \Sp_2(\mu) \rightarrow 0.$$  
	Thus, the representation $\chi \times \Sp_2(\mu)$ is $\Sp_3(\D)$-distinguished by Lemma \ref{9}.
	
	(3) If $\Dim(\sigma_1')=1$ and $\Dim(\sigma_1'')=1$, then $\sigma_1 = \sigma_1' \times \sigma_1''$ is  $\Sp_2(\D)$-distinguished. Now, we have two  sub-cases depending on the representation $\sigma_2$.
	
	(3a) If $\Dim(\sigma_2)>1$, then both the orbits do not participate. Hence, $\pi$ is not $\Sp_3(\D)$-distinguished in this case.
	
	(3b)	If $\Dim(\sigma_2)=1$ and either $\sigma_1' \simeq \sigma_2 \otimes \nu$ or $\sigma_1'' \simeq \sigma_2 \otimes \nu$,  then the closed orbit participates to an $\Sp_3(\D)$-invariant linear form. Hence, $\pi$ is $\Sp_3(\D)$-distinguished.
	If $\Dim(\sigma_2)=1$ and neither $\sigma_1' \simeq \sigma_2 \otimes \nu$ nor $\sigma_1'' \simeq \sigma_2 \otimes \nu$,  then the closed orbit does not participate. By Theorem  \ref{27} and Lemma \ref{11},
	$$\Ext_{\Sp_3(\D)}^1(\Ind_{P_{H}}^{\Sp_3(\D)} \nu^{3}[(\sigma_1 \otimes \sigma_2)|_{M_{H}} , \mathbb{C}) = 0.$$
	Hence, by Proposition \ref{26}, $\pi = \sigma_1' \times \sigma_1'' \times \sigma_2$ has a symplectic period arising from the open orbit.

	We now analyze the subquotients of the representation $\pi = \sigma_1' \times \sigma_1'' \times \sigma_2$ (Recall that $\Dim(\sigma_1') = \Dim(\sigma_1'') = \Dim(\sigma_2) = 1$).
	If none of the pairs are linked, then $\pi$ is irreducible and $\Sp_3(\D)$-distinguished.
	If there is precisely one linked pair among three, then no new $\Sp_3(\D)$-distinguished subquotient is obtained.
	If there are precisely two linked pairs among three, then $\pi$ is one of the following representations:
	\begin{itemize}
		\item[\upshape(i)] $\chi\nu_1^{-1} \times \chi\nu_1^{-1} \times \chi,$\vspace{2mm}
		\item[\upshape(ii)] $\chi \times \chi \times \chi\nu_1^{-1},$\vspace{2mm}
		\item[\upshape(iii)] $\chi\nu_1^{-1} \times \chi\nu_1 \times \chi.$
	\end{itemize}
	
	If $\pi$ is either $\chi\nu_1^{-1} \times \chi\nu_1^{-1}  \times \chi$ or $\chi \times \chi \times \chi\nu_1^{-1}$, then no new $\Sp_3(\D)$-distinguished subquotient is obtained.
	If $\pi = \chi\nu_1^{-1} \times \chi\nu_1 \times \chi$, then by Theorem $\ref{8}$, $\pi$ has four irreducible subquotients 
	\begin{itemize}
		\item  $\theta_1 = L([\nu_1^{-1} \chi , \nu_1 \chi]),$\vspace{2mm}
		\item $\theta_2 = L([\nu_1 \chi],[\nu_1^{-1} \chi , \chi]),$\vspace{2mm}
		\item $\theta_3 = L([\nu_1^{-1}\chi],[\chi , \nu_1 \chi]),$\vspace{2mm}
		\item  $\theta_4 = L([\nu_1^{-1}\chi],[\nu_1 \chi], [\chi]).$
	\end{itemize}
	
	Of these, $\theta_1$ is a unique irreducible quotient of  $L([\nu_1^{-1} \chi , \chi]) \times \nu_1 \chi$. By the above Case $5$, the representation $L([\nu_1^{-1} \chi , \chi]) \times \nu_1 \chi$ is not $\Sp_3(\D)$-distinguished. It implies that the representation $\theta_1$ is not $\Sp_3(\D)$-distinguished.
	
	The representation $\theta_2$ is isomorphic to $\tau = Z([\nu_1^{-1}\chi],[\chi , \nu_1 \chi])$.
	The representation $Z([\chi , \nu_1 \chi]) \times \nu_1^{-1}\chi$ is glued from two irreducible  representations  $\tau$ and $Z([\nu_1^{-1} \chi , \nu_1 \chi])$.
	The latter representation  $Z([\nu_1^{-1} \chi , \nu_1 \chi]) \simeq \theta_4$ is a unitary character of $\GL_3(\D)$ and has a symplectic period. By Proposition \ref{25}, $Z([\chi , \nu_1 \chi]) \times \nu_1^{-1}\chi$ has a symplectic period and $\Dim \Hom_{\Sp_3(\D)}(Z([\chi , \nu_1 \chi]) \times \nu_1^{-1}\chi|_{\Sp_3(\D)} , \mathbb{C}) = 1$. It follows that the irreducible representation  $\tau \simeq \theta_2$ has no symplectic period. By taking contragredient and using a similar argument, we can conclude that $\theta_3$ also does not have a symplectic period.
	\section{Distinguished representations of $\GL_4(\D)$}
	This section analyzes the irreducible admissible representations of $\GL_4(\D)$ with a symplectic period and proves Theorem \ref{3}.
	
	\subsection{Necessary and sufficient conditions on distinction for a  representation of $\GL_4(\D)$}  
	\subsubsection{} By Conjecture $7.1$ and Proposition $7.3$ in \cite{Verma}, the supercuspidal representations of $\GL_4(\D)$ are not $\Sp_4(\D)$-distinguished. Therefore, it is enough to study the problem of the symplectic period for the non-supercuspidal representations of $\GL_4(\D)$.
	We begin with the following lemma. 
	\begin{lem}\label{29}
		Let $\theta$ be a non-supercuspidal irreducible admissible representation of $\GL_4(\D)$ with a symplectic period. Then either $\theta$ appears as a quotient of $\pi_1 \times \pi_2$, where $\pi_1 , \pi_2 \in \Irr(\GL_2(\D))$ or $\theta = \lambda_1 \times \lambda_2$, where $\lambda_1 \in \Irr(\GL_1(\D))$ and $\lambda_2$ is irreducible supercuspidal representation of $\GL_3(\D)$. 
	\end{lem}
	\begin{proof}
		If $\theta$ is a non-supercuspidal irreducible admissible representation of $\GL_4(\D)$, then $\theta$ appears as a quotient of either $\lambda_1 \times \lambda_2$, $\lambda_2 \times \lambda_1$, or $\pi_1 \times \pi_2$, where $\lambda_1 \in \Irr(\GL_1(\D))$, $\lambda_2 \in \Irr(\GL_3(\D))$, and $\pi_1 , \pi_2 \in \Irr(\GL_2(\D))$. The last case is trivial.
		
		If $\theta$ is a quotient of $\lambda_2 \times \lambda_1$, then $\tilde{\theta}$ is a quotient of $\tilde{\lambda_1} \times \tilde{\lambda_2}$. By Lemma \ref{10}, again we are reduced to the first case. So assume that $\theta$ is a quotient of $\lambda_1 \times \lambda_2$. Now if $\lambda_1 \in \Irr(\GL_1(\D))$  and $\lambda_2$ is an irreducible supercuspidal representation of $\GL_3(\D)$, then again we have nothing left to prove. Thus, assume $\lambda_2$ is not supercuspidal. Hence, $\lambda_2$ is a quotient of one of the representations of the form $\tau_1 \times \tau_2$, $\tau_2 \times \tau_1$, or $\sigma_1 \times \sigma_2 \times \sigma_3$, where  $\tau_1$ is an irreducible representation of  $\GL_2(\D)$, $\tau_2 \in \Irr(\GL_1(\D))$ and $\sigma_1 , \sigma_2 , \sigma_3 \in \Irr(\GL_1(\D))$.
		
		In the first case, $\lambda_1 \times \lambda_2$ is a quotient of $  \lambda_1 \times \tau_1 \times \tau_2$. Since $\tau_1 \times \tau_2$ is irreducible, $\lambda_1 \times \tau_1 \times \tau_2 \simeq \lambda_1 \times \tau_2 \times \tau_1$. If $\lambda_1 \times \tau_2$ is irreducible, then the lemma is proved. If not, $\lambda_1 \times \tau_2 \times \tau_1$ is glued from $\langle \lambda_1 \times \tau_2 \rangle \times \tau_1$ and  $\langle \lambda_1 \times \tau_2 \rangle^t \times \tau_1$, where $\langle \lambda_1 \times \tau_2 \rangle $ and $\langle \lambda_1 \times \tau_2 \rangle^t$ are the unique irreducible subrepresentation and unique irreducible quotient of $\lambda_1 \times \tau_2$. Thus, any irreducible quotient of $\lambda_1 \times \lambda_2$ has to be a quotient of one of the two.
		
		In the second case, since $\tau_1 \times \tau_2$ is irreducible, $ \lambda_1 \times \tau_2 \times \tau_1 \simeq \lambda_1 \times \tau_1 \times \tau_2$. Thus, we are back to the first case.
		
		The case, where the representations $\lambda_1 \times \sigma_1$ and $\sigma_2 \times \sigma_3$ are irreducible, is trivial. In the case where at least one of them is reducible, we get the lemma by breaking $\lambda_1 \times \sigma_1 \times \sigma_2 \times \sigma_3$, as in the first case, into subquotients of the required form.       
	\end{proof}
	
	\noindent
	By Lemma \ref{29}, we first apply Mackey theory to the representations of $\GL_4(\D)$ parabolically induced from the maximal parabolic subgroup $P_{1,3}$ and after that, on the representations of $\GL_4(\D)$ parabolically induced from the maximal parabolic subgroup $P_{2,2}$. 
	\subsubsection{}\label{30} Consider the representation $\pi = \lambda_1 \times \lambda_2 := \Ind_{P_{1,3}}^{\GL_4(\D)}\lambda$ of $\GL_4(\D)$, where $\lambda = \lambda_1 \otimes \lambda_2$ is an irreducible representation of $\GL_1(\D) \times \GL_3(\D)$. Analyzing the restriction of $\pi$ to $\Sp_4(\D)$, we obtain the following exact sequence of $\Sp_4(\D)$-representations by applying Mackey theory and  Proposition \ref{15}: 
    $$0 \rightarrow \ind_{H_{1,0}}^{\Sp_4(\D)}[(\lambda_1 \otimes \lambda_2)|_{M_{1,0}}] \rightarrow \pi \rightarrow  \Ind_{H_{1,1}}^{\Sp_4(\D)}\delta_{H_{1,1}}^{1/2}[(\lambda_1 \otimes \lambda_2)|_{M_{1,1}}] \rightarrow 0,$$
	where $\delta_{H_{1,1}}$ is the character on $H_{1,1}$ given by 
	\begin{equation*}
		\delta_{H_{1,1}} \Bigg{[}\begin{pmatrix}
			a & b & c & d \\
			0 & e & f & g  \\ 
			0 & h & i & j  \\
			0 & 0 & 0 & \overline{a}^{-1} 
		\end{pmatrix}\Bigg{]} = |\N_{\D/\K}(a)|_{\K}^{8}.
	\end{equation*}
	\noindent
	It is easy to observe by Proposition $1.9$(b) in \cite{BZ} that 
	$$\Hom_{\Sp_4(\D)}(ind_{H_{1,s}}^{\Sp_4(\D)}[\delta_P^{1/2}(\lambda_1 \otimes \lambda_2)|_{M_{1,s}}] , \mathbb{C}) = \Hom_{M_{1,s}}(r_{(s,~1-s)}(\lambda_1) \otimes r_{(3-s,~s)}(\lambda_2), \nu),$$
	for $s = 0,1$. Consequently, the following propositions outline the necessary and sufficient conditions for distinction in a parabolically induced representation of $\GL_4(\D)$.

	\begin{prop} \label{32}
		If the representation 
		$$\pi = \lambda_1 \times \lambda_2  := \Ind_{P_{1,3}}^{\GL_4(\D)}(\lambda_1 \otimes \lambda_2)$$
		has an $\Sp_4(\D)$-invariant form, then either
		\begin{enumerate}
			\item[\upshape(1)]  $\Hom_{\Delta(\GL_1(\D) \times \GL_1(\D)) \times \Sp_2(\D)}(\nu^{-1}[\lambda_1 \otimes r_{(2,1),(3)}(\lambda_2)]~,\mathbb{C}) \not= 0$~(\text{closed orbit condition}) or 
			
			\item[\upshape(2)]   $\lambda_1$ and $\lambda_2$ have a  symplectic period~(\text{open orbit condition}).
		\end{enumerate}
	\end{prop}
	\begin{prop}\label{hariom}
		The representation $$\pi = \lambda_1 \times \lambda_2  := \Ind_{P_{1,3}}^{\GL_4(\D)}(\lambda_1 \otimes \lambda_2)$$
		has an $\Sp_4(\D)$-invariant form,
		if one of the following holds
		\begin{enumerate}
			\item[\upshape(1)] $\Hom_{\Delta(\GL_1(\D) \times \GL_1(\D)) \times \Sp_2(\D)}(\nu^{-1}[\lambda_1 \otimes r_{(2,1),(3)}(\lambda_2)]~,\mathbb{C}) \not= 0$,
			\item[\upshape(2)]   $\lambda_1$ and $\lambda_2$ have a symplectic period with 
			$$\Ext_{\Sp_4(\D)}^1(\Ind_{H_{1,1}}^{\Sp_4(\D)}
			\nu^4[(\lambda_1 \otimes \lambda_2)|_{M_{1,1}}] , \mathbb{C}) = 0.$$
		\end{enumerate} 
	\end{prop}
	\subsubsection{} We now analyze the functor 
	$$\mathcal{P}_2 = \Ext_{\Sp_4(\D)}^1(\Ind_{H_{1,1}}^{\Sp_4(\D)}\nu^4[(\lambda_1 \otimes \lambda_2)|_{M_{1,1}}] , \mathbb{C})$$ 
	mentioned in the Proposition $\ref{hariom}$.
	Recall that $\lambda_1 \in \Irr(\GL_1(\D))$, $\lambda_2 \in \Irr(\GL_{3}(\D))$, and $$H_{1,1} = M_{1,1} U_{1,1} = (\Delta(\GL_1(\D) \times \GL_1(\D)) \times \Sp_2(\D)) U_{1,1} ~ (\text{see Proposition}~ \ref{15}).$$
	Assume $r_{(2,1),(3)}(\lambda_2) = \sum_{i=1}^{t}(\lambda_{2i} \otimes \lambda_{2i}')$.
	For $i = 1,2,\ldots,t$, define 
	$$\mathcal{S}_i = \Hom_{\GL_1(\D)}(\lambda_1 ,\lambda_{2i}' \otimes \nu ) \otimes \Ext_{\Sp_2(\D)}^1(\lambda_{2i} , \mathbb{C})$$ and $$\mathcal{T}_i = \Ext_{\GL_1(\D)}^1(\lambda_1 ,\lambda_{2i}' \otimes \nu) \otimes \Hom_{\Sp_2(\D)}(\lambda_{2i} , \mathbb{C}).$$
	Then following theorem follows by the same technique as Theorem $\ref{27}$. Therefore, the proof is omitted.
	\begin{thm}\label{34}
		The functor $\mathcal{P}_2$ can be expressed as a product of the direct sum of the functors $\mathcal{S}_i$ and $\mathcal{T}_i$, i.e., 
		$$\mathcal{P}_2 = \prod_{i=1}^t (\mathcal{S}_i \oplus \mathcal{T}_i).$$
	\end{thm}

	\subsubsection{}\label{35} By Lemma $\ref{29}$, we investigate the presence of a symplectic period in each quotient of the representations of the form $\pi = \lambda_1 \times \lambda_2$, where $\lambda_1 \in \Irr(\GL_1(\D))$ and $\lambda_2$ is irreducible supercuspidal representation of $\GL_3(\D)$. Proposition \ref{32} implies that the part of $\pi$ supported on the closed orbit does not participate in $\Sp_4(\D)$-invariant linear form. As a result, two sub-cases arise:\\
	\textbf{Case 1:} In the case, where one of $\lambda_1$ or $\lambda_2$ does not have a symplectic period, then the open orbit also does not participate. So, $\pi$ does not have a symplectic period by Proposition \ref{32}.\\ 
	\textbf{Case 2:} If $\pi = \lambda_1 \times \lambda_2$, where $\lambda_1 $ is a character of $\GL_1(\D)$ and $\lambda_2$ is irreducible supercuspidal representation of $\GL_3(\D)$ with a symplectic period,  then the open orbit participates. Since $\lambda_2$ is an irreducible supercuspidal representation of $\GL_3(\D)$, by Theorem \ref{34} we have $$\Ext_{\Sp_4(\D)}^1(\Ind_{H_{1,1}}^{\Sp_4(\D)} 
	\nu^4[(\lambda_1 \otimes \lambda_2)|_{M_{1,1}}] , \mathbb{C}) = 0.$$ 
	Hence, the irreducible representation $\pi =  \lambda_1 \times \lambda_2$ has a symplectic period arising from the open orbit by Proposition \ref{hariom}.
	\subsubsection{}\label{54} By Lemma \ref{29}, we now consider the representations of the form $\pi = \pi_1 \times \pi_2 := \Ind_{P_{2,2}}^{\GL_4(\D)}\sigma$ of $\GL_4(\D)$, where $\sigma = \pi_1 \otimes \pi_2$ is an irreducible representation of $\GL_2(\D) \times \GL_2(\D)$. Analyzing the restriction of $\pi$ to $\Sp_4(\D)$ using Mackey theory and Proposition \ref{15}, we obtain that $(\pi_1 \times \pi_2)|_{\Sp_4(\D)}$ is glued from the three subquotients,
	$$\ind_{H_{2,0}}^{\Sp_4(\D)} (\pi_1 \otimes \pi_2),~ \ind_{H_{2,1}}^{\Sp_4(\D)}\delta_{H_{2,1}}^{1/2} (\pi_1 \otimes \pi_2)|_{M_{2,1}}, ~\text{and}~  \ind_{H_{2,2}}^{\Sp_4(\D)}\delta_{H_{2,2}}^{1/2} (\pi_1 \otimes \pi_2)|_{M_{2,2}},$$
	where $\delta_{H_{2,2}}$ and $\delta_{H_{2,1}}$ are the characters on $H_{2,2}$ and $H_{2,1}$, respectively; given by 
	$$\delta_{H_{2,2}} \Bigg{[}\begin{pmatrix}
		a & *  \\
		
		0 & \overline{a}^{-1} 
	\end{pmatrix}\Bigg{]} = |\N_{\D/\K}(a)|^8 , a \in \GL_2(\D),$$
	$$\delta_{H_{2,1}} \Bigg{[}\begin{pmatrix}
		a & * & * & * \\
		0 & b & * & * \\ 
		0 & 0 & c & * \\
		0 & 0 & 0 & \overline{a}^{-1} 
	\end{pmatrix}\Bigg{]} = |\N_{\D/\K}(a)|^8, a \in \GL_1(\D).$$

	\noindent
	Through an analysis similar to $\S$\ref{23} of these three subquotients, the subsequent proposition is derived, thus omitting the proof.
	\begin{prop} \label{36}
		If the representation 
		$$\pi = \pi_1 \times \pi_2 := \Ind_{P_{2,2}}^{\GL_4(\D)}(\pi_1 \otimes \pi_2)$$
		has an $\Sp_4(\D)$-invariant form, then either
		\begin{enumerate}
			\item[\upshape(1)] $\pi_1 \simeq \nu \pi_2$ (\text{Closed orbit condition}) or
			
			\item[\upshape(2)] $\pi_1$ and $\pi_2$ are $\Sp_2(\D)$-distinguished or
			
			\item[\upshape(3)]  $ \Hom_{M_{2,1}}(\nu^{-1} (r_{(1,1),(2)} (\pi_1) \otimes r_{(1,1),(2)} (\pi_2)) ,  \mathbb{C} ) \not= 0.$ 
		\end{enumerate}
	\end{prop}

	\subsubsection{}\label{37}  If $\pi = \pi_1 \times \pi_2$ has an $\Sp_4(\D)$-distinguished quotient, then $\pi$ itself is $\Sp_4(\D)$-distinguished. 
	If either of $\pi_1$ or $\pi_2$ is supercuspidal, then by  Proposition \ref{36} and Lemma \ref{13}, the representation $\pi = \pi_1 \times \pi_2$ has a symplectic period arising from closed orbit if and only if $\pi_1 \simeq \nu \pi_2$.
	
	If $\pi = \nu \pi_2 \times \pi_2$, where $\pi_2$ is an irreducible supercuspidal representation of $\GL_2(\D)$, then $\pi$ has a symplectic period arising from the closed orbit. By Theorem \ref{7}, the representation $\pi$ has a unique irreducible subrepresentation $L([\pi_2, \nu \pi_2])$ and a unique irreducible quotient $L([\nu \pi_2],[ \pi_2])$. The contragredient  of $L([\pi_2, \nu \pi_2])$ is equal to $L([\nu^{-1} \tilde{\pi_2} ,   \tilde{\pi_2}])$. It is a quotient of the representation $\nu^{-1}  \tilde{\pi_2} \times  \tilde{\pi_2}$ which does not have a symplectic period by Proposition \ref{36}. It implies that the representation $L([\nu^{-1}  \tilde{\pi_2} ,   \tilde{\pi_2}])$ also does not have a symplectic period. Hence, by Lemma \ref{10} we conclude that the  representation $L([\pi_2 , \nu \pi_2])$ does not have a symplectic period. Thus, the  representation $\theta = L([\nu \pi_2],[ \pi_2])  \simeq Z([\pi_2 , \nu\pi_2])$ has a symplectic period by Lemma \ref{9}.
	\subsection{Representations of $\GL_4(\D)$ with a symplectic period}
	\subsubsection{}\label{39}
	By $\S$\ref{35}, $\S$\ref{37}, and Proposition \ref{36}, any irreducible  $\Sp_4(\D)$-distinguished representation occurs as a quotient of one of the representations listed in the following proposition. 
	
	\begin{prop}\label{40}
		
		Let $\theta$ be an irreducible admissible representation of $\GL_4(\D)$ with a symplectic period. Then either $\theta = Z([\pi_2,\nu \pi_2])$, where $\pi_2$ is a supercuspidal representation of $\GL_2(\D)$ or $\theta = \lambda_1 \times \lambda_2$, where $\lambda_1 $ is a character of $\GL_1(\D)$ and $\lambda_2$ is an irreducible supercuspidal representation of $\GL_3(\D)$ with a symplectic period, or it occurs as a quotient of one of the following representations $\pi$ of $\GL_4(\D)$:
		\begin{enumerate}
			\item[\upshape(1)]  $\pi = \pi_1 \times \nu^{-1} \pi_1$, where $\pi_1$ is a non-supercuspidal irreducible representation of $\GL_2(\D)$.
			
			\item[\upshape(2)]  $\pi = \pi_1 \times \pi_2$, where $\pi_1$, $\pi_2$ are either an irreducible principal series $\sigma_1 \times \sigma_2$ with $\Dim (\sigma_1 \otimes \sigma_2) = 1$ or a character or a Speh representation.
			
			\item[\upshape(3)]  $\pi = \pi_1 \times \pi_2$ such that $\Hom_{M_{2,1}}(\nu^{-1} [r_{(1,1),(2)} (\pi_1) \otimes r_{(1,1),(2)} (\pi_2)] ,  \mathbb{C} ) \not= 0$, 
			where $\pi_1 , \pi_2$ are non-supercuspidal irreducible representations of $\GL_2(\D)$.
		\end{enumerate}
	\end{prop}

	\subsubsection{}\label{41}
	By Proposition \ref{40}, it is enough to consider and  check  all irreducible quotients or subquotients (up to permutation) of each representation in the following list:
	\begin{enumerate}
		\item   $\pi = \sigma_1 \times \sigma_2 \times \sigma_1' \times \sigma_2'$, where $\sigma_1 \times \sigma_2$ , $\sigma_1' \times \sigma_2'$ are irreducible principal series representations with $\Dim (\sigma_1 \otimes \sigma_2) = 1 = \Dim (\sigma_1' \otimes \sigma_2')$.
		
		\item $\pi = \sigma_1 \times \sigma_2 \times \chi_2$, where $\sigma_1 \times \sigma_2$ is an irreducible principal series representation of $\GL_2(\D)$ with $\Dim (\sigma_1 \otimes \sigma_2) = 1$, and $\chi_2$ is a character of $\GL_2(\D)$.
		
		\item $\pi = \sigma_1 \times \sigma_2 \times \Sp_2(\mu)$, where  $\sigma_1 \times \sigma_2$ is an irreducible principal series representation of $\GL_2(\D)$ with $\Dim (\sigma_1 \otimes \sigma_2) = 1$, and $\Sp_2(\mu)$ is the Speh representation of $\GL_2(\D)$.
		
		\item $\pi = \chi_2 \times \Sp_2(\mu)$, where $\chi_2$ and $\Sp_2(\mu)$ are character and the Speh representation of $\GL_2(\D)$, respectively.
		
		\item $\pi = \chi_2 \times \chi_2'$, where $\chi_2 , \chi_2'$ are characters of $\GL_2(\D)$.

		\item  $\pi = \Sp_2(\mu_2) \times \Sp_2(\mu_2')$, where $\Sp_2(\mu_2) , \Sp_2(\mu_2')$ are Speh representations of $\GL_2(\D)$.
		
		\item  $\pi = \nu^{1/2}\mu \times \chi_1 \times \chi_1' \times \nu^{-1/2}\mu$, where $\chi_1 , \chi_1'$ are characters of $\GL_1(\D)$ and $\mu$ is an irreducible representation of $\GL_1(\D)$ with $\Dim(\mu)>1$.
		
		\item $\pi = \sigma_1 \times \chi_1 \times \chi_2 St_2$, where $\chi_2St_2$ is a quotient of  $\chi_1\nu^{-1} \times \chi_1\nu$ and $\sigma_1 \times \chi_1$ is an irreducible principal series representation of $\GL_2(\D)$ with $\Dim(\sigma_1 \otimes \chi_1) = 1$.
		
		\item $\pi = \chi_2 \times \nu^{-1}\chi_2St_2$, where $\chi_2$,  $\chi_2 St_2$ are subrepresentation and quotient of  $\chi_1\nu^{-1} \times \chi_1\nu$;  $\chi_1$ is a character of $\GL_1(\D)$.
		
		\item $\pi = \chi_2St_2 \times \nu\chi_2St_2$, where $\chi_2 St_2$ is quotient of $\chi_1\nu^{-1} \times \chi_1\nu$;  $\chi_1$ is a character of $\GL_1(\D)$.
		
		\item $\pi = \pi_1 \times \nu^{-1}\pi_1$, where $\pi_1$ is either a  generalized Steinberg representation $St_2(\mu)$ of $\GL_2(\D)$ or an irreducible principal series representation $\sigma_1 \times \sigma_2$ of $\GL_2(\D)$ with $\Dim (\sigma_1 )$, $ \Dim (\sigma_2) > 1$.
	\end{enumerate}
	
	Now, we consider each case individually and find all distinguished subquotients which prove Theorem \ref{3}.\\
	\textbf{Case 1:} If  $\pi = \sigma_1 \times \sigma_2 \times \sigma_1' \times \sigma_2'$, where $\sigma_1 \times \sigma_2$, $\sigma_1' \times \sigma_2'$ are irreducible principal series with $\Dim (\sigma_1 \otimes \sigma_2) = 1 = \Dim (\sigma_1' \otimes \sigma_2')$, then we have two conditions for the distinction of the representation $\pi$ by Proposition \ref{32}.
	If either $\sigma_1 \simeq \sigma_2' \otimes \nu$ or $\sigma_1 \simeq \sigma_1' \otimes \nu$ or $\sigma_1 \simeq \sigma_2 \otimes \nu$, then $\pi$ has a symplectic period arising from the closed orbit.
	If neither $\sigma_1 \simeq \sigma_2' \otimes \nu$ nor $\sigma_1 \simeq \sigma_1' \otimes \nu$ nor $\sigma_1 \simeq \sigma_2 \otimes \nu$, then by Theorem \ref{34} and Lemma \ref{11}, $$\Ext_{\Sp_4(\D)}^1(\Ind_{H_{1,1}}^{\Sp_4(\D)}\nu^4[\sigma_1 \otimes (\sigma_2 \times \sigma_1' \times \sigma_2')]|_{M_{1,1}} , \mathbb{C}) = 0.$$ Therefore,  $\pi$ has a $\Sp_4(\D)$-invariant linear form arising from the open orbit by Proposition \ref{hariom}.
	Now, we analyze the presence of the symplectic period in subquotients of $\pi$ by dividing it into five sub-cases.
	
	(1) If none of the pairs are linked, then $\pi$ is irreducible and $\Sp_4(\D)$-distinguished.
	
	(2) If there is precisely one linked pair among four, say $\sigma_2$, $\sigma_1'$, then $\pi$ is of the form  
	$$\sigma_1 \times \chi_1 \nu_1^{-1/2} \times \chi_1 \nu_1^{1/2} \times \sigma_2'.$$ 
	The representation $\pi$ is glued from two irreducible representations $\theta_1 = \sigma_1 \times \chi_2 \times \sigma_2'$ and $\theta_2 = \sigma_1 \times \chi_2 St_2 \times \sigma_2'$.
	Proposition \ref{32} implies that $\theta_1$ is $\Sp_4(\D)$-distinguished, and $\theta_2$ is $\Sp_4(\D)$-distinguished if and only if either $\sigma_1 \simeq \chi_1$ or $\sigma_2' \simeq \chi_1$.
	
	(3) If there are precisely two linked pairs among four, then $\pi$ is either
	\begin{itemize}
		\item[\upshape(i)] $\pi_1 = \sigma_1 \times \nu_1^{-1/2} \chi_1 \times \nu_1^{1/2} \chi_1 \times \nu_1^{1/2} \chi_1$ or\vspace{2mm}
		\item[\upshape(ii)] $\pi_2 = \sigma_1 \times \nu_1^{1/2} \chi_1  \times \nu_1^{-1/2} \chi_1  \times \nu_1^{-1/2} \chi_1$ or\vspace{2mm}
		\item[\upshape(iii)] $\pi_3 = \sigma_1 \times \nu_1^{1/2} \chi_1  \times \nu_1^{-1/2} \chi_1  \times \nu_1^{3/2} \chi_1.$
	\end{itemize}
	If $\pi$ is either $\pi_1$ or $\pi_2$, then no  new $\Sp_4(\D)$-distinguished subquotients are obtained.
	If $\pi = \pi_3 = \sigma_1 \times \nu_1^{1/2} \chi_1  \times \nu_1^{-1/2} \chi_1  \times \nu_1^{3/2} \chi_1$, then it has four irreducible subquotients 
	
	\begin{itemize}
		\item $\theta_1 = \sigma_1 \times L([\nu_1^{-1/2} \chi_1, \nu_1^{3/2} \chi_1])$,\vspace{2mm}
		
		\item  $\theta_2 = \sigma_1 \times L([\nu_1^{-1/2} \chi_1,\nu_1^{1/2} \chi_1],[\nu_1^{3/2} \chi_1]),$ \vspace{2mm}
		
		\item  $\theta_3 = \sigma_1 \times L([\nu_1^{1/2} \chi_1,\nu_1^{3/2} \chi_1],[\nu_1^{-1/2} \chi_1]),$\vspace{2mm}
		
		\item  $\theta_4 = \sigma_1 \times L([\nu_1^{1/2} \chi_1],[\nu_1^{-1/2} \chi_1],[\nu_1^{3/2} \chi_1]).$
	\end{itemize}

	It follows from Proposition \ref{32} that $\theta_1$ has no symplectic period, and $\theta_4$ has a  symplectic period. Again by Proposition \ref{32},
	$\theta_2$ has a symplectic period arising from  the closed orbit if and only if $\sigma_1 \simeq \chi_1$, and 
	$\theta_3$ has a symplectic period arising from the closed orbit if and only if $\sigma_1 \simeq  \nu^2\chi_1$.
	
	(4) If there are precisely three linked pairs among four, then $\pi$  is either
	
	\begin{itemize}
		\item[\upshape(i)] $\pi_1 = \nu_1^{-1/2} \chi_1 \times \nu_1^{3/2} \chi_1 \times \nu_1^{-3/2} \chi_1 \times \nu_1^{1/2} \chi_1$ or\vspace{2mm}
		\item[\upshape(ii)] $\pi_2 = \nu_1^{-3/2} \chi_1 \times \nu_1^{1/2} \chi_1 \times \nu_1^{-1/2} \chi_1 \times \nu_1^{3/2} \chi_1$.
	\end{itemize}

	\noindent
	If $\pi =\pi_1$, then it is isomorphic to $\nu_1^{-1/2} \chi_1 \times \nu_1^{-3/2} \chi_1 \times \nu_1^{3/2} \chi_1 \times \nu_1^{1/2} \chi_1$. This representation has $\tau_1 = Z([\nu_1^{-3/2} \chi_1 , \nu_1^{-1/2} \chi_1]) \times Z([\nu_1^{1/2} \chi_1 , \nu_1^{3/2} \chi_1])$ as a quotient and so any  irreducible $\Sp_4(\D)$-distinguished  quotient $\theta$ of $\pi_1$ is a quotient of $\tau_1$. Therefore, we get
	$$\theta =  Z([\nu_1^{-3/2} \chi_1 , \nu_1^{-1/2} \chi_1], [\nu_1^{1/2} \chi_1 , \nu_1^{3/2} \chi_1]).$$
	By Proposition \ref{36} and using a similar argument as in the Case $4$ in $\S$\ref{28}, $\theta$ is not $\Sp_4(\D)$-distinguished.
	
	If $\pi = \pi_2$, then it has $\tau_2 = Z([\nu_1^{1/2} \chi_1], [\nu_1^{3/2} \chi_1]) \times Z([\nu_1^{-3/2} \chi_1] , [\nu_1^{-1/2} \chi_1])$  as a quotient and so any irreducible $\Sp_4(\D)$-distinguished  quotient $\theta$ of $\pi_2$ is a quotient of $\tau_2$. By Proposition \ref{36},  the representation $\tau_2$ is not $\Sp_4(\D)$-distinguished. So, $\theta$ is also not $\Sp_4(\D)$-distinguished.
	
	(5) If four pairs are linked, then up to permutation $\pi$ is  one of the following representations: 
	\begin{itemize}
		\item[\upshape(i)] $\pi_1 = \nu_1^{1/2}\chi_1 \times \nu_1^{1/2} \chi_1 \times \nu_1^{-1/2} \chi_1   \times \nu_1^{-1/2} \chi_1$,\vspace{2mm}
		\item[\upshape(ii)] $\pi_2 = \nu_1^{-1/2} \chi_1 \times \nu_1^{3/2} \chi_1 \times \nu_1^{1/2} \chi_1 \times \nu_1^{1/2} \chi_1 $.
	\end{itemize}
	
	If $\pi = \pi_1$, then it has three irreducible  subquotients
	\begin{itemize}
		\item $\theta_1 = L([\nu_1^{-1/2} \chi_1 , \nu_1^{1/2} \chi_1],[\nu_1^{-1/2} \chi_1 , \nu_1^{1/2} \chi_1]),$\vspace{2mm}
		\item $\theta_2 = L([\nu_1^{1/2}\chi_1],[\nu_1^{-1/2}\chi_1,\nu_1^{1/2} \chi_1],[\nu_1^{-1/2} \chi_1])$, \vspace{2mm}
		\item $\theta_3 = L([\nu_1^{1/2}\chi_1 ],[\nu_1^{1/2}\chi_1],[\nu_1^{-1/2}\chi_1],[\nu_1^{-1/2} \chi_1])$.
	\end{itemize}
	
	\noindent  
	By  Proposition \ref{36}, we can conclude  that  $\theta_1 \simeq L([\nu_1^{-1/2} \chi_1 , \nu_1^{1/2} \chi_1]) \times L([\nu_1^{-1/2} \chi_1 , \nu_1^{1/2} \chi_1])$ does not have a symplectic period, and $\theta_3 \simeq Z([\nu_1^{-1/2}\chi_1,\nu_1^{1/2} \chi_1]) \times Z([\nu_1^{-1/2}\chi_1,\nu_1^{1/2} \chi_1])$  has a symplectic period. 
	The representation $\theta_2$ is the unique irreducible quotient of $\nu_1^{1/2}\chi_1 \times \nu_1^{-1/2} \chi_1 \times L([\nu_1^{-1/2}\chi_1,\nu_1^{1/2} \chi_1])$. Since $L([\nu_1^{1/2}\chi_1],[\nu_1^{-1/2}\chi_1])$ is a quotient of $\nu_1^{1/2}\chi_1 \times \nu_1^{-1/2}\chi_1$, $\theta_2$ is also the unique irreducible quotient of $\tau = L([\nu_1^{1/2}\chi_1],[\nu_1^{-1/2}\chi_1]) \times L([\nu_1^{-1/2}\chi_1,\nu_1^{1/2} \chi_1])$. The representation $\tau$ does not have a symplectic period by Proposition \ref{36}. It implies that $\theta_2$ has no symplectic period.

	If $\pi = \pi_2$, then it has six irreducible subquotients. By using a similar argument as above, we can conclude that one of the subquotient $$L([\nu_1^{-1/2} \chi_1],[\nu_1^{3/2} \chi_1],[\nu_1^{1/2} \chi_1],[\nu_1^{1/2} \chi_1])$$ is $\Sp_4(\D)$-distinguished and others are not distinguished.\\
	\textbf{Case 2:}~ If $\pi = \sigma_1 \times \sigma_2 \times \chi_2$, where $\sigma_1 \times \sigma_2$ is an irreducible principal series with $\Dim (\sigma_1 \otimes \sigma_2) = 1$, and $\chi_2$ is a character of $\GL_2(\D)$, then  we have two conditions for the distinction of $\pi$ by Proposition \ref{32}.
	If either $\sigma_1 \simeq \sigma_2 \otimes \nu$ or $\sigma_1 \simeq \nu^2 \chi_1$, then $\pi$ has a symplectic period arising from the closed orbit.
	If neither  $\sigma_1 \simeq \sigma_2 \otimes \nu$ nor $\sigma_1 \simeq \nu^2 \chi_1$, then by Theorem \ref{34} and Lemma \ref{11}, $$\Ext_{\Sp_4(\D)}^1(\Ind_{P_{H}}^{\Sp_4(\D)}\nu^4[\sigma_1 \otimes (\sigma_2 \times \chi_2)]|_{M_{H}} , \mathbb{C}) = 0.$$  Hence, it follows from Proposition \ref{hariom} that $\pi$ has a symplectic period arising from the open orbit.
	Now, we analyze the subquotients of $\pi$ by dividing it into three sub-cases, whether or not they have a symplectic period.
	
	(1) If none of the pairs are linked, then $\pi$ is irreducible and has a symplectic period.
	
	(2) If there is precisely one linked pair among three, then assume $\sigma_2 = \nu_1^{3/2} \chi_1$ without loss of generality. The representation $\pi$ is glued from two irreducible subquotients $\theta_1 = \sigma_1 \times Z([\nu_1^{3/2} \chi_1],[\nu_1^{-1/2} \chi_1 , \nu_1^{1/2} \chi_1])$ and $\theta_2 = \sigma_1 \times Z([\nu_1^{-1/2} \chi_1 , \nu_1^{3/2} \chi_1])$.
	Proposition \ref{32} implies that $\theta_2$ has a symplectic period, and $\theta_1$ has a symplectic period if and only if $\sigma_1 \simeq \nu_1 \chi_1$.
	
	(3) If there are precisely two linked pairs among three, then $\pi$ is either
	
	\begin{itemize}
		\item[\upshape(i)] $\pi_1 = \nu_1^{3/2}\chi_1 \times \nu_1^{3/2}\chi_1 \times Z([\nu_1^{-1/2} \chi_1 , \nu_1^{1/2} \chi_1])$ or \vspace{2mm}
		\item[\upshape(ii)]  $\pi_2 = \nu_1^{-3/2}\chi_1 \times \nu_1^{-3/2}\chi_1 \times  Z([\nu_1^{-1/2} \chi_1 , \nu_1^{1/2} \chi_1])$ or \vspace{2mm}
		\item[\upshape(iii)]  $\pi_3 = \nu_1^{3/2}\chi_1 \times \nu_1^{-3/2}\chi_1 \times Z([\nu_1^{-1/2} \chi_1 , \nu_1^{1/2} \chi_1])$.
	\end{itemize}
	
	If $\pi$ is  either $\pi_1$ or $\pi_2$,  then no new $\Sp_4(\D)$-distinguished subquotient is obtained. If $\pi = \pi_3$, then it has four irreducible subquotients
	
	\begin{itemize}
		\item $\theta_1 = Z([\nu_1^{-3/2}\chi_1 , \nu_1^{3/2}\chi_1])$,\vspace{2mm}
		
		\item $\theta_2 = Z([\nu_1^{-3/2}\chi_1], [\nu_1^{-1/2} \chi_1 , \nu_1^{3/2}\chi_1])$,\vspace{2mm}

		\item $\theta_3 = Z([\nu_1^{3/2}\chi_1],[\nu_1^{-3/2}\chi_1 , \nu_1^{1/2} \chi_1])$,\vspace{2mm}

		\item $\theta_4 = Z([\nu_1^{3/2}\chi_1],[\nu_1^{-3/2}\chi_1],[\nu_1^{-1/2} \chi_1 , \nu_1^{1/2} \chi_1])$.
	\end{itemize}
	
	The representation $\tau_1 = \nu_1^{-3/2}\chi_1 \times Z([\nu_1^{-1/2} \chi_1 , \nu_1^{3/2}\chi_1])$ is glued from two irreducible representations $\theta_1$ and $\theta_2$. The representation $\theta_1$ is a character of $\GL_4(\D)$ and has a symplectic period.
	By Proposition \ref{32}, the representation $\tau_1$ has a symplectic period and  $\Dim \Hom_{\Sp_4(\D)}(\tau_1|_{\Sp_4(\D)} , \mathbb{C}) = 1$. It implies that the  representation $\theta_2$ does not have a symplectic period, and by taking contragradient we conclude that neither does $\theta_3$.
	
	The representation $\theta_4$  is the unique irreducible quotient of $\nu_1^{-3/2}\chi_1 \times Z ([\nu_1^{-1/2} \chi_1 , \nu_1^{1/2} \chi_1]) \times \nu_1^{3/2}\chi_1$. Since  $Z([\nu_1^{-1/2} \chi_1 , \nu_1^{1/2} \chi_1],[\nu_1^{3/2}\chi_1])$ is  a quotient of $Z( [\nu_1^{-1/2} \chi_1 , \nu_1^{1/2} \chi_1]) \times \nu_1^{3/2}\chi_1$, $\theta_4$ is also a unique irreducible quotient of $\tau_2 = \nu_1^{-3/2}\chi_1 \times Z([\nu_1^{-1/2} \chi_1 , \nu_1^{1/2} \chi_1],[\nu_1^{3/2}\chi_1])$. Proposition \ref{32} implies that the representation $\tau_2$ does not have a symplectic period. Hence, $\theta_4$ does not have a symplectic period.\\
	\textbf{Case 3:} If $\pi = \sigma_1 \times \sigma_2 \times \Sp_2(\mu)$, where $\sigma_1 \times \sigma_2$ is an irreducible principal series representation with $\Dim(\sigma_1 \otimes \sigma_2) = 1$,  and $\Sp_2(\mu)$ is the Speh representation of $\GL_2(\D)$, then we have two conditions for the distinction of the representation $\pi$ by Proposition \ref{32}.
	If  $\sigma_1 \simeq \sigma_2 \otimes \nu$, then the irreducible representation $\pi$ has a symplectic period arising from the closed orbit.
	If  $\sigma_1 \not\simeq \sigma_2 \otimes \nu$, then the closed orbit does not contribute in $\Sp_4(\D)$-invariant linear form. However,  the open orbit does. In this case, Theorem \ref{34} and Lemma \ref{11} suggest that $$\Ext_{\Sp_4(\D)}^1(\Ind_{P_{H}}^{\Sp_4(\D)}\nu^4[\sigma_1 \otimes (\sigma_2 \times \Sp_2(\mu))]|_{M_{H}} , \mathbb{C}) = 0.$$  So, the irreducible representation $\pi$ has a symplectic period by Proposition \ref{hariom}.\\ 
	\textbf{Case 4:} If $\pi =  \chi_2 \times \Sp_2(\mu)$, where $\chi_2$ is a character of $\GL_2(\D)$ and $\Sp_2(\mu)$ is the Speh representation of $\GL_2(\D)$, then it is an irreducible subquotient of the representation $\tau = \nu^{-1}\chi_1 \times \nu\chi_1 \times \nu^{1/2}\mu \times \nu^{-1/2}\mu$.  The representation $\tau$ has four irreducible subquotients $\chi_2 \times \Sp_2(\mu)$, $\chi_2 \times St_2(\mu)$, $\chi_2 St_2 \times \Sp_2(\mu)$, and $\chi_2 St_2 \times St_2(\mu)$. By Proposition \ref{32}, $\tau$ has a symplectic period arising from the open orbit and the representations  $\chi_2 \times St_2(\mu)$, $\chi_2 St_2 \times \Sp_2(\mu)$, and $\chi_2 St_2 \times St_2(\mu)$ have no symplectic period by Proposition \ref{36}. Therefore, by Lemma \ref{9} we deduce that the representation $\chi_2 \times \Sp_2(\mu)$ has a symplectic period. \\
	\textbf{Case 5:} If $\pi = \chi_2 \times \chi_2' $, where $\chi_2 , \chi_2'$ are characters of $\GL_2(\D)$, then it is an irreducible subquotient of the representation $\tau = \nu^{-1}\chi_1 \times \nu\chi_1 \times \nu^{-1}\chi_1' \times \nu\chi_1'$. The representation $\tau$ has four irreducible subquotients $\chi_2 \times \chi_2'$, $\chi_2 \times \chi_2'St_2$, $\chi_2St_2 \times \chi_2'$, and $\chi_2St_2 \times \chi_2'St_2$. Now, we have two conditions for the distinction of the representation $\pi$. 
	
	If either $\chi_1 \simeq \chi_1' \nu^{-1}$ or $\chi_1 \simeq \chi_1' \nu$, then by taking contragredient and by  Proposition \ref{36}, the representation $\pi$ has a symplectic period arising from the closed orbit.
	If neither $\chi_1 \simeq \chi_1' \nu^{-1}$  nor  $\chi_1 \simeq \chi_1' \nu$, then by  Proposition \ref{36},  the representations $\chi_2 \times \chi_2'St_2$, $\chi_2St_2 \times \chi_2'$, and $\chi_2St_2 \times \chi_2'St_2$ have no symplectic period and the representation $\tau$ has a symplectic period by  Proposition \ref{32}. Hence, by Lemma \ref{9} we conclude that $\pi$ has a symplectic period.

	Assume $\chi_2 = Z([\chi_1\nu_1^{-1/2} , \chi_1\nu_1^{1/2}])$ and $\chi_2' = Z([\chi_1'\nu_1^{-1/2} , \chi_1'\nu_1^{1/2}])$, where $\chi_1$ and $\chi_1'$ are characters of $\GL_1(\D)$. If $\pi$ is irreducible, then it has a symplectic period.
	Otherwise, we have the following two sub-cases: 
	
	(1) If $\chi_1 = \chi_1' \nu_1$, then the representation $\pi = Z([\chi_1'\nu_1^{1/2} , \chi_1'\nu_1^{3/2}]) \times Z([\chi_1'\nu_1^{-1/2} , \chi_1'\nu_1^{1/2}])$ \linebreak is glued from two irreducible representations  $\theta_1 = Z([\chi_1'\nu_1^{1/2} , \chi_1'\nu_1^{3/2}] , [\chi_1'\nu_1^{-1/2} , \chi_1'\nu_1^{1/2}])$ and \linebreak $\theta_2 = Z([\chi_1'\nu_1^{-1/2} , \chi_1'\nu_1^{3/2}]) \times \chi_1'\nu_1^{1/2}$. By Proposition \ref{32}, $\theta_2$ has a symplectic period and by Proposition \ref{36}, $\Dim \Hom_{\Sp_4(\D)}(\pi|_{\Sp_4(\D)} , \mathbb{C}) = 1$. It implies that $\theta_1$ has no symplectic period.
	
	(2) If $\chi_1 = \chi_1' \nu_1^2$, then the representation $\pi = Z([\chi_1'\nu_1^{3/2} , \chi_1'\nu_1^{5/2}]) \times Z([\chi_1'\nu_1^{-1/2} , \chi_1'\nu_1^{1/2} ])$  is glued from two irreducible representations  $\theta_1 = Z([\chi_1'\nu_1^{3/2} , \chi_1'\nu_1^{5/2}] , [\chi_1'\nu_1^{-1/2} , \chi_1'\nu_1^{1/2}])$ and  $\theta_2 = Z([\chi_1'\nu_1^{-1/2} , \chi_1'\nu_1^{5/2}])$. The  representation $\theta_2$ is a unitary character of $\GL_4(\D)$ and has a symplectic period. We obtain by  Proposition \ref{36} that $\Dim \Hom_{\Sp_4(\D)}(\pi|_{\Sp_4(\D)} , \mathbb{C}) = 1$. Therefore, $\theta_1$ has no symplectic period.\\  
	\textbf{Case 6:}  If $\pi = \Sp_2(\mu_1) \times \Sp_2(\mu_1')$, where $\Sp_2(\mu_1) , \Sp_2(\mu_1')$ are Speh representations of $\GL_2(\D)$, then it is an irreducible subquotient of the representation $\tau = \nu^{1/2}\mu_1 \times \nu^{-1/2}\mu_1 \times \nu^{1/2}\mu_1' \times \nu^{-1/2}\mu_1'$. The representation $\tau$ has four irreducible subquotients $\Sp_2(\mu_1) \times \Sp_2(\mu_1')$, $\Sp_2(\mu_1) \times St_2(\mu_1')$, $St_2(\mu_1) \times \Sp_2(\mu_1')$, and $St_2(\mu_1) \times St_2(\mu_1')$. Now, we have two conditions for the distinction of the representation $\pi$.
	
	If $\mu_1 \simeq \mu_1' \nu$, then  by Proposition \ref{36} the representation $\pi$ has an $\Sp_4(\D)$-invariant linear form arising from the closed orbit.
	If $\mu_1 \not\simeq \mu_1' \nu$, then by Proposition \ref{36},   the representations $\Sp_2(\mu_1) \times St_2(\mu_1')$, $St_2(\mu_1) \times \Sp_2(\mu_1')$, and $St_2(\mu_1) \times St_2(\mu_1')$ has no symplectic period, and by Proposition \ref{32} the representation $\tau$ has a symplectic period arising from the closed orbit. Therefore,  $\pi$ has a symplectic period by Lemma \ref{9}.

	Assume $\Sp_2(\mu_1) = Z([\mu_1\nu^{-1/2} , \mu_1\nu^{1/2}])$ and $\Sp_2(\mu_1') = Z([\mu_1'\nu^{-1/2} , \mu_1'\nu^{1/2}])$, where $\mu_1 , \mu_1'$ are representations of $\GL_1(\D)$ with $\Dim(\mu_1 ), \Dim(\mu_1' ) > 1 $. If $\pi$ is irreducible, then it has a symplectic period. Otherwise, the following two sub-cases arise:
	
	(1) If $\mu_1 = \mu_1' \nu$,	then $\pi = Z([\mu_1'\nu^{1/2} , \mu_1'\nu^{3/2}]) \times Z([\mu_1'\nu^{-1/2} , \mu_1'\nu^{1/2}])$ is glued from two irreducible representations $\theta_1 = Z([\mu_1'\nu^{1/2} , \mu_1'\nu^{3/2}] , [\mu_1'\nu^{-1/2} , \mu_1'\nu^{1/2}])$ and $\theta_2 = Z([\mu_1'\nu^{-1/2} , \mu_1'\nu^{3/2}]$ $) \times \mu_1'\nu^{1/2}.$ Proposition \ref{32} implies that the representation $\theta_2$ does not have a symplectic period. Hence, $\theta_1$ has a symplectic period by Lemma \ref{9}.
	
	(2) If $\mu_1 = \mu_1' \nu^2$,	then $\pi = Z([\mu_1'\nu^{3/2} , \mu_1'\nu^{5/2}]) \times Z([\mu_1'\nu^{-1/2} , \mu_1'\nu^{1/2}])$ is glued from two irreducible representations $\theta_1 = Z([\mu_1'\nu^{3/2} , \mu_1'\nu^{5/2}] , [\mu_1'\nu^{-1/2} , \mu_1'\nu^{1/2}])$ and $\theta_2 = Z([\mu_1'\nu^{-1/2} , \mu_1'\nu^{5/2}]$ $)$. By using the algorithm of Moeglin and Waldspurger \cite{badulescu2007zelevinsky}, we find that the representation $\theta_1$ is isomorphic to the representation $\tau = L([\mu_1'\nu^{5/2}] , [\mu_1'\nu^{1/2} ,  \mu_1'\nu^{3/2}] , [\mu_1'\nu^{-1/2}])$ which is a unique irreducible quotient of $\mu_1'\nu^{5/2} \times L([\mu_1'\nu^{1/2} , \mu_1'\nu^{3/2}]) \times \mu_1'\nu^{-1/2}$. Since the representation  $L([\mu_1'\nu^{1/2} , \mu_1'\nu^{3/2}],[\mu_1'\nu^{-1/2}])$ is a quotient of  $L([\mu_1'\nu^{1/2} , \mu_1'\nu^{3/2}]) \times \mu_1'\nu^{-1/2}$,  $\tau$ is the unique irreducible quotient of $\mu_1'\nu^{5/2} \times L([\mu_1'\nu^{1/2} , \mu_1'\nu^{3/2}],[\mu_1'\nu^{-1/2}])$. By Proposition \ref{32}, the representation $\mu_1'\nu^{5/2} \times L([\mu_1'\nu^{1/2} , \mu_1'\nu^{3/2}],[\mu_1'\nu^{-1/2}])$ does not have a symplectic period. So, $\tau \simeq \theta_1$ also does not have a symplectic period. Thus,  $\theta_2$ has a symplectic period by Lemma \ref{9}.\\
	\textbf{Case 7:} If  $\pi = \nu^{1/2}\mu \times \chi_1 \times \chi_1' \times \nu^{-1/2}\mu$, where $\chi_1 , \chi_1'$ are characters of $\GL_1(\D)$ and $\mu$ is an irreducible representation of $\GL_1(\D)$ with $\Dim(\mu) >1$, then we have two sub-cases:
	
	(1) If $\chi_1$ and $\chi_1'$ are not  linked, then $\pi$ has two irreducible subquotients $\theta_1 = \chi_1 \times \chi_1' \times \Sp_2(\mu)$ and $\theta_2 = \chi_1 \times \chi_1' \times St_2 (\mu)$. By Proposition \ref{32}, $\theta_1$ is $\Sp_4(\D)$-distinguished and 
	$\theta_2$ is not $\Sp_4(\D)$-distinguished.
	
	(2) If $\chi_1$ and $\chi_1'$ are linked, then $\pi$ has 4 irreducible subquotients $\chi_2 \times \Sp_2(\mu) $, $\chi_2 St_2 \times \Sp_2(\mu)$, $\chi_2 \times St_2 (\mu)$, and $\chi_2 St_2 \times St_2 (\mu)$.
	By  Proposition \ref{36}, the irreducible representations $\chi_2 St_2 \times \Sp_2(\mu)$, $\chi_2 \times St_2 (\mu)$, and $\chi_2 St_2 \times \times St_2 (\mu)$ are not $\Sp_4(\D)$-distinguished and the representation $\pi$ is $\Sp_4(\D)$-distinguished by  Proposition \ref{32}. Hence, by Lemma \ref{9} we deduce that the representation $\chi_2 \times \Sp_2(\mu)$ is $\Sp_4(\D)$-distinguished.\\ 
	\textbf{Case 8:} If $\pi = \sigma_1 \times \chi_1 \times \chi_2 St_2$, where $\sigma_1 \times \chi_1$ is an irreducible principal series with $\Dim(\sigma_1 \otimes \chi_1) = 1$, and $\chi_2St_2$ is a quotient of $\chi_1\nu^{-1} \times \chi_1\nu$, then by Proposition \ref{32} the  representation $\pi = \sigma_1 \times \chi_1 \times \chi_2 St_2$ has a symplectic period. If $\pi$ is irreducible, then it has a symplectic period. Otherwise,
	without loss of generality, it is glued from two irreducible subquotients $\theta_1 = \chi_1 \times L([\chi_1\nu_1^{-1/2},\chi_1\nu_1^{3/2}])$ and $\theta_2 = \chi_1 \times L([\chi_1\nu_1^{3/2}],[\chi_1\nu_1^{-1/2},\chi_1\nu_1^{1/2}])$.
	Proposition \ref{32} implies that $\theta_1$ does not have a symplectic period and $\theta_2$ does.\\ 
	\textbf{Case 9:} If $\pi = \chi_2 \times \nu^{-1}\chi_2St_2$, where $\chi_2$,  $\chi_2 St_2$ are subrepresentation and quotient of the representation $\chi_1\nu^{-1} \times \chi_1\nu$; $\chi_1$ is a character of $\GL_1(\D)$, then $\tilde{\pi} = \tilde{\chi_2} \times \nu \tilde{\chi_2} St_2$, which does not have a symplectic period  by  Proposition \ref{36}. Therefore, by Lemma \ref{10} we conclude that the irreducible representation $\pi$ has no symplectic period.\\ 
	\textbf{Case 10:} If $\pi = \chi_2St_2 \times \nu\chi_2St_2$, where $\chi_2 St_2$ is a quotient of $\chi_1\nu^{-1} \times \chi_1\nu$;  $\chi_1$ is a character of $\GL_1(\D)$, then by Proposition \ref{36} $\tilde{\pi} = \tilde{\chi_2} St_2 \times \nu^{-1} \tilde{\chi_2} St_2$ has a symplectic period arising from  the closed orbit . Hence, the irreducible representation $\pi$ has a symplectic period by Lemma \ref{10}.\\
	\textbf{Case 11:} Let $\pi = \pi_1 \times \nu^{-1}\pi_1$, where  $\pi_1$ is either generalized Steinberg $St_2(\mu)$ of $\GL_2(\D)$ or an irreducible principal series $\sigma_1 \times \sigma_2$  of $\GL_2(\D)$ with $\Dim(\sigma_1)$, $\Dim(\sigma_2) > 1$.
	
	If $\pi_1$ is generalized Steinberg $St_2(\mu)$, then  $\pi = L([\nu^{-1/2}\mu,\nu^{1/2}\mu]) \times L([\nu^{-3/2}\mu , \nu^{-1/2}\mu])$ is glued from  two irreducible representations $\theta_1 = L([\nu^{-3/2}\mu , \nu^{1/2}\mu])$ $\times \nu^{-1/2}\mu$ and  $\theta_2 =  L([\nu^{-1/2}\mu,\nu^{1/2}\mu] ,$ $[\nu^{-3/2}\mu , \nu^{-1/2}\mu])$. By Proposition \ref{32}, we deduce that the  representation $\theta_1$ has no symplectic period, and Proposition \ref{36} implies the presence of a symplectic  period in $\pi$  arising from the closed orbit. Thus, $\theta_2$ has a symplectic period by Lemma \ref{10}.
	
	If $\pi_1$ is an irreducible principal series $\sigma_1 \times \sigma_2$ with  $\Dim(\sigma_1 )$,  $\Dim(\sigma_2) > 1$, then $\pi = \sigma_1 \times \sigma_2 \times \nu^{-1} \sigma_1 \times \nu^{-1} \sigma_2$. As a result, two sub-cases arise:
	
	(1) If $\sigma_2 \times \nu^{-1} \sigma_1$ is irreducible, then we have two conditions.
	
	If $\sigma_2 \not= \sigma_1 \nu^2$, then $\pi \simeq \sigma_1 \times \nu^{-1} \sigma_1 \times \sigma_2 \times \nu^{-1} \sigma_2 $ has a unique irreducible quotient $\theta = Z([\nu^{-1}\sigma_1 , \sigma_1]) \times Z([\nu^{-1}\sigma_2 , \sigma_2])$. Using the similar argument as in above Case $6$, we deduce that $\theta$ carries a symplectic period.
	
	If $\sigma_2 = \sigma_1 \nu^2$, then $\pi = \sigma_2 \nu^{-2} \times \sigma_2 \times \sigma_2 \nu^{-3} \times \sigma_2 \nu^{-1} \simeq \sigma_2\nu^{-2} \times \sigma_2 \nu^{-3} \times \sigma_2 \times \sigma_2 \nu^{-1}$. This representation has $\tau = Z([\sigma_2 \nu^{-3} , \sigma_2\nu^{-2}]) \times Z([\sigma_2 \nu^{-1} , \sigma_2 ])$ as a quotient. Since the cuspidal support of $\pi$ is multiplicity free, it has a unique irreducible quotient, and so any irreducible $\Sp_4(\D)$-distinguished quotient of $\pi$ is a quotient of $\tau$. Thus, they have already been accounted for in the above Case $6$.
	
	(2) If $\sigma_2 \times \nu^{-1} \sigma_1$ is reducible, then this condition holds if and only if   either $\sigma_1 = \sigma_2 \nu^2$ or  $\sigma_1 = \sigma_2$. 
	If $\sigma_1 = \sigma_2 \nu^2$, then $\pi \simeq \sigma_2  \times \sigma_2 \nu^{-1} \times \sigma_2 \nu^2 \times  \sigma_2 \nu$. In this case, no new $\Sp_4(\D)$-distinguished subquotients is obtained.
	If $\sigma_1 = \sigma_2$, then $\pi \simeq \sigma_1 \times \sigma_1 \times \sigma_1 \nu^{-1} \times \sigma_1 \nu^{-1}$ has a unique irreducible quotient  $\theta = Z([ \sigma_1 \nu^{-1} , \sigma_1 ]) \times Z([\sigma_1 \nu^{-1} , \sigma_1 ])$. Using the similar argument as in above Case $6$, we get that the representation $\theta$ has a symplectic period.

	\subsection{Acknowledgments} We thank Dipendra Prasad for carefully reading certain proofs and valuable suggestions regarding the modulus characters. We also thank Mayank Singh for carefully reading the article. The first author acknowledges the research support from the ``Council of Scientific and Industrial Research (CSIR), Bharat (India)" with File No. 09/143(1005)/2019-EMR-I. The second author is partially supported by the SERB, MTR/2021/000655.

	\bibliography{bibliography}

\def\cprime{$'$}
\begin{thebibliography}{MVW87}

\bibitem[BR07]{badulescu2007zelevinsky}
A.~I. Badulescu and D.~Renard.
\newblock Zelevinsky involution and {M}oeglin-{W}aldspurger algorithm for
  {${\rm GL}_n(\rm D)$}.
\newblock {\em Funct. Anal. IX}, 48:9--15, 2007.

\bibitem[BZ77]{BZ}
I.~N. Bernstein and A.~V. Zelevinsky.
\newblock Induced representations of reductive {$\mathfrak p$-adic groups. I}.
\newblock {\em Ann. Sci. Ecole Norm. Sup.}, 10:441--472, 1977.

\bibitem[HR90]{HR90}
M.~J. Heumos and S.~Rallis.
\newblock Symplectic-{W}hittaker models for {${\rm GL}_n$}.
\newblock {\em Pacific J. Math.}, 146:247--279, 1990.

\bibitem[Kly84]{AA84}
A.~A. Klyachko.
\newblock Models for complex representations of the groups {${\rm GL}(n,\,q)$}.
\newblock {\em Math. USSR-Sb.}, 48:365--379, 1984.

\bibitem[Mit14]{mitra2014representations}
A.~Mitra.
\newblock On representations of ${{\rm GL}_{2n} (F)}$ with a symplectic period.
\newblock {\em Pacific J. Math.}, 268:435--463, 2014.

\bibitem[MOS17]{mitra2017klyachko}
A.~Mitra, O.~Offen, and E.~Sayag.
\newblock Klyachko models for ladder representations.
\newblock {\em Doc. Math.}, 22:611--657, 2017.

\bibitem[MS13]{minguez2013representations}
A.~Minguez and V.~S{\'e}cherre.
\newblock Repr{\'e}sentations banales de.
\newblock {\em Compos. Math.}, 149:679--704, 2013.

\bibitem[MVW87]{MVW}
C.~M{\oe}glin, M.-F. Vign{\'e}ras, and J.-L Waldspurger.
\newblock {\em Correspondances de {H}owe sur un corps {$p$}-adique}, volume
  1291 of {\em Lecture Notes in Math.}
\newblock 1987.

\bibitem[OS07a]{OSE09}
O.~Offen and E.~Sayag.
\newblock Global mixed periods and local {K}lyachko models for the general
  linear group.
\newblock {\em Int. Math. Res. Not. IMRN}, 2007:rnm136, 2007.

\bibitem[OS07b]{OSE07}
O.~Offen and E.~Sayag.
\newblock On unitary representations of {${\rm GL}_{2n}$} distinguished by the
  symplectic group.
\newblock {\em J. Number Theory}, 125:344--355, 2007.

\bibitem[OS08]{OSE08}
O.~Offen and E.~Sayag.
\newblock Uniqueness and disjointness of {K}lyachko models.
\newblock {\em J. Funct. Anal.}, 254:2846--2865, 2008.

\bibitem[Pra18]{prasad2013ext}
D.~Prasad.
\newblock Ext-analogues of branching laws.
\newblock {\em Proceedings of the {I}nternational {C}ongress of
  {M}athematicians - {R}io de {J}aneiro}, II:1367--1392, 2018.

\bibitem[Rag07]{raghuram5541}
A.~Raghuram.
\newblock On the restriction to {$\rm{D}^* \times \rm{D}^*$} of representations
  of $\mathfrak p$-adic {$\rm {GL}_2(\rm {D})$}.
\newblock {\em Canad. J. Math.}, 59:1050--1068, 2007.

\bibitem[Tad90]{tadic1990induced}
M.~Tadic.
\newblock Induced representations of {$\rm {GL} (n, \rm{A})$} for $\mathfrak
  p$-adic division algebras {$\rm{A}$}.
\newblock {\em J. Reine Angew. Math.}, 405:48--77, 1990.

\bibitem[Ven13]{Venket}
C.~G. Venketasubramanian.
\newblock On representations of {${\rm GL}(n)$} distinguished by {${\rm
  GL}(n-1)$} over a {$p$}-adic field.
\newblock {\em Israel J. Math.}, 194:1--44, 2013.

\bibitem[Ver18]{Verma}
M.~K. Verma.
\newblock On symplectic periods for inner forms of {${\rm GL}_n$}.
\newblock {\em Math. Research Letters}, 25:309--334, 2018.

\end{thebibliography}
	\bibliographystyle{alpha}

\end{document}